\documentclass[a4paper,11pt,reqno]{amsart}
 \usepackage{enumerate}
 \usepackage{a4wide}
 \usepackage{amssymb, amsmath}
 \usepackage{mathrsfs}
 \usepackage{amscd}
 \usepackage[active]{srcltx}
 \usepackage{verbatim}
 \usepackage[colorlinks,linkcolor={red},citecolor={blue},urlcolor={black}]{hyperref}
 \usepackage[normalem]{ulem}

\theoremstyle{plain}
\newtheorem{theorem}{Theorem}[section]
\newtheorem{corollary}[theorem]{Corollary}
\newtheorem{lemma}[theorem]{Lemma}
\newtheorem{proposition}[theorem]{Proposition}

\newtheorem{definition}[theorem]{Definition}

\newtheorem*{definition*}{Definition}
\newtheorem{conjecture}[theorem]{Conjecture}

\theoremstyle{remark}
\newtheorem{remark}[theorem]{Remark}

\newtheorem*{claim*}{Claim}
\newtheorem*{remark*}{Remark}
\newtheorem*{example*}{Example}
\newtheorem*{notation*}{Notation}

\numberwithin{equation}{section}

\DeclareMathOperator{\Var}{Var}

\DeclareMathOperator{\Ent}{Ent}

\def\N{{\mathbb N}}
\def\Z{{\mathbb Z}}
\def\R{{\mathbb R}}

\newcommand{\one}{{{\bf 1}}}

\renewcommand{\d}{\delta}
\newcommand{\eps}{\varepsilon}
\renewcommand{\phi}{\varphi}

\newcommand{\dd}{\; \mathrm{d}}

\newcommand{\ip}[1]{\langle {#1}\rangle}

\newcommand{\norm}[1]{\| {#1}\|}
\newcommand{\abs}[1]{\vert {#1}\vert}

\DeclareMathOperator{\Ric}{Ric}

\newcommand{\ddt}{\frac{\mathrm{d}}{\mathrm{d}t}}

\newcommand{\cH}{\mathcal{H}}
\newcommand{\cB}{\mathcal{B}}

\newcommand{\cW}{\mathcal{W}}
\newcommand{\cI}{\mathcal{I}}

\newcommand{\cG}{\mathcal{G}}

\newcommand{\cX}{\mathcal{X}}
\newcommand{\cE}{\mathcal{E}}
\newcommand{\cA}{\mathcal{A}}

\newcommand{\cP}{\mathscr{P}}
\newcommand{\PX}{\cP(\cX)}
\newcommand{\PXs}{\cP_*(\cX)}
\newcommand{\hrho}{\hat\rho}

\renewcommand{\tilde}{\widetilde}

\newcommand{\e}{\mathrm{e}}

\begin{document}

\title[Functional inequalities for Markov chains with non-negative curvature]
{Poincar\'e, modified logarithmic Sobolev and isoperimetric inequalities for Markov chains with non-negative Ricci curvature}

\author{Matthias Erbar}
\address{
University of Bonn\\
Institute for Applied Mathematics\\
Endenicher Allee 60\\
53115 Bonn\\
Germany}
\email{erbar@iam.uni-bonn.de}

\author{Max Fathi}
\address{
University of California, Berkeley\\
Department of Mathematics\\
Evans Hall, Berkeley\\
USA}
\email{maxf@berkeley.edu}



\begin{abstract}
  We study functional inequalities for Markov chains on discrete
  spaces with entropic Ricci curvature bounded from below. Our main
  results are that when curvature is non-negative, but not necessarily
  positive, the spectral gap, the Cheeger isoperimetric constant and
  the modified logarithmic Sobolev constant of the chain can be
  bounded from below by a constant that only depends on the diameter
  of the space, with respect to a suitable metric. These estimates are
  discrete analogues of classical results of Riemannian geometry
  obtained by Li and Yau, Buser and Wang.
\end{abstract}



\maketitle
 

\section{Introduction}
\label{sec:intro}

Ricci curvature bounds play an important role in geometric analysis on
Riemannian manifolds. For instance, a lower bound on the curvature by
a strictly positive constant entails many interesting properties for
the manifold, most notably Harnack inequalities, bounds on the
eigenvalues of the Laplacian, concentration bounds and isoperimetric
inequalities.

In the light of this wide range of implications, considerable effort
has been put into developing a notion of Ricci curvature lower bounds
for non-smooth spaces. Bakry and \'Emery \cite{BE85} proposed a
curvature condition for general Markov diffusion operators via the
so-called $\Gamma$-calculus. Lott-Villani \cite{LV} and Sturm
\cite{St06} presented an approach that applies to (geodesic) metric
measure spaces. Such a space has Ricci curvature bounded below by a
constant $\kappa$ provided the entropy is $\kappa$-convex along
geodesics in the Wasserstein space of probability measures.
Subsequently, many of the classical relating curvature bounds to
functional inequalities have been generalized to such 'continuous'
non-smooth spaces, we refer to \cite{BGL15, Vi08} for an overview.

In recent years, there has been a strong interest in developing an
analogous theory for discrete spaces. Unfortunately, the
Lott--Sturm--Villani theory does not apply and a number of alternative
notions of Ricci bounds have been proposed, see for instance
\cite{BS,GRST14,Oll09}. In this work, we will focus on the notion of
\emph{entropic Ricci curvature bounds} put forward in \cite{Ma12,EM12}
that applies to finite Markov chains and seems to be particularly well
suited to study discrete functional inequalities. Here the key point
is to replace the $L^2$-Wasserstein distance with a new transportation
distance $\cW$ in the definition of Lott--Sturm--Villani. It has been
shown in \cite{EM12} that a strictly positive entropic Ricci curvature
lower bound implies a spectral gap estimate, a modified logarithmic
Sobolev inequality and an analogue of Talagrand's transport cost
inequality.

In the present work, we are interested in the situation where the
curvature is bounded from below but not strictly positive. We show
that in this situation relatively weak extra information (for instance
a bound on the diameter of the space) still allows one to establish
strong functional inequalities.

To state our main results we consider an irreducible and reversible
continuous time Markov chain on a finite space $\cX$ whose generator
is given by
\begin{align*}
  L\psi (x) = \sum\limits{y\in\cX}\big(\psi(y)-\psi(x)\big)Q(x,y)\;,
\end{align*}
where $Q(x,y)$ are the transition rates between $x$ and $y$ and let
$\pi$ be the unique reversible probability measure.

For the purpose of this introduction we state our main results for
simplicity under the assumption that the chain has non-negative
entropic Ricci curvature. We shall actually derive more general
statements allowing for a negative curvature bound in the main
text. We refer to Section \ref{sec:prelim} for a precise definition of
entropic Ricci curvature bounds and the functional inequalities we
consider.

The first result establishes an isoperimetric inequality using
information on the spectral gap (see Theorem \ref{thm:isoperimetric} below).

\begin{theorem} \label{main_thm_buser} If the entropic Ricci curvature
  of $(\cX,Q,\pi)$ is non-negative, then the Cheeger (or linear
  isoperimetric) constant $h$ and the spectral gap $\lambda_1$ of $L$
  satisfy
  \begin{align*}
    h \geq \frac{1}{3}\sqrt{Q_* \lambda_1}
  \end{align*}
where $Q_*=\min\big(Q(x,y):Q(x,y)>0\big)$ is the minimal transition
rate. Here the Cheeger constant is defined by
 \begin{align*}
   h = \max\limits_{A\subset\cX}\frac{\pi^+(\partial A)}{\pi(A)(1-\pi(A))}\;,
 \end{align*}
 where $\pi^+(\partial A)=\sum_{x\in A,y\in A^c}Q(x,y)\pi(x)$ denotes the
 perimeter measure of $A$.
\end{theorem}

This result is a discrete version of the classical Buser theorem in
Riemannian geometry \cite{Bus}. A simple analytic proof was later
obtained by Ledoux \cite{Led94}, and extended to weighted spaces in
\cite{Led04}. A matching upper bound (up to a different universal
constant) is valid in any space, without any assumption on the
curvature.

The next two results establish estimates on the spectral gap and the
logarithmic Sobolev constant in non-negative curvature using
information on the diameter of $\cX$. A natural distance $d_\cW$ on
$\cX$ is induced by the discrete transport distance $\cW$ between
probability measures by setting
$d_{\mathcal{W}}(x,y) := \mathcal{W}(\delta_x, \delta_y)$. The
distance $d_\cW$ can be compared to more traditional weighted graph
distances, see Lemma \ref{lem:distance-comp}, yielding immediate
analogues of the results below in terms of weighted graph distance.

\begin{theorem} \label{main_thm_sg} If the entropic Ricci curvature of
  $(\cX,Q,\pi)$ is non-negative and the diameter of
  $(\cX, d_{\mathcal{W}})$ is bounded by $D$, then the spectral gap of
  of the generator $L$ satisfies
$$\lambda_1 \geq \frac{c}{D^2}$$
for some universal constant $c$.
\end{theorem}

See Theorem \ref{thm:Poincare} below for a more general statement in
negative curvature. The continuous version of this statement is a classical result of Li
and Yau \cite{LY} (extending a previous result of Li \cite{Li79} for
manifolds without boundary and of \cite{PW60} for convex sets in
Euclidean spaces), and for which the sharp constant was determined in
\cite{ZY}. A version taking into account the dimension has been
obtained by Bakry and Qian \cite{BQ}, and recently extended to
geodesic metric measure spaces by Cavaletti and Mondino \cite{CM1,
  CM2} (with a completely different method).

\begin{theorem} \label{main_thm_mlsi} If the entropic Ricci curvature
  of the Markov chain is non-negative and the diameter of
  $(X, d_{\mathcal{W}})$ is bounded by $D$, then a modified
  logarithmic Sobolev inequality holds, with constant $\frac{c'}{D^2}$
  for some universal constant $c'$.
\end{theorem}

See Theorem \ref{thm:MLSI} below. This last result is a weakened
discrete version of a work of Wang \cite{Wan97}, where the finite
diameter is replaced by a bound on some square-exponential moment of
the distance to an arbitrarily fixed point. It immediately implies a
discrete version of Talagrand's inequality, via the discrete
Otto-Villani theorem of \cite{EM12}, as well as an upper bound on the
total variation mixing time, as we shall see in Section 5.2. We shall
actually obtain a version of this result only assuming finiteness of a
square-exponential moment as in \cite{Wan97} but with a constant that
we do not believe to be sharp, see Theorem \ref{thm_mlsi_conc} below.

In the last two results, the dependence on the diameter is optimal,
since it is sharp (up to the values of the constants $c$ and $c'$) for
the random walk on the one-dimensional discrete torus. However, it
behaves badly in high dimensions. This leads us to formulate
conjectures about possible improvements using measure concentration
bounds instead of diameter bounds in Section 5. Impressive results in
this direction for manifolds have been obtained by Milman
\cite{Mil09,Mil10}.

Versions of Theorems \ref{main_thm_buser} and \ref{main_thm_sg} have
been obtained for another notion of curvature, namely a discrete
version of the Bakry-\'Emery $\Gamma_2$ condition, in \cite{KKRT14}
and \cite{CLY} respectively. The two notions of curvature are known to
be \emph{not} equivalent. A Markov chain with non-negative entropic
Ricci curvature but negative Bakry-\'Emery curvature has recently been
discovered \cite{EFMST}. However, no analogue of Theorem
\ref{main_thm_mlsi} is known using Bakry-\'Emery curvature. To our
knowledge it is not even known whether strictly positive Bakry-\'Emery
curvature is enough to ensure the validity of a modified logarithmic
Sobolev inequality as in Theorem \ref{main_thm_mlsi}. The proofs of
Theorem \ref{main_thm_buser} and the main result of \cite{KKRT14} are
quite close and both based on arguments developed by Ledoux in the
continuous setting. For Theorem \ref{main_thm_sg}, we shall give two
proofs. One of them replicates the technique used in \cite{CLY}. The
other one uses an HWI interpolation inequality obtained in
\cite{EM12}, for which no analogue is known in the setting of discrete
Bakry-\'Emery curvature. This technique has the advantage that the
assumptions can be weakened to a bound on a square exponential moment
instead of the diameter. It will also be used to prove Theorem
\ref{main_thm_mlsi}.

One of the main technical tools in our study is a new equivalent
characterization of entropic Ricci curvature lower bounds in terms of
gradient estimates for the associated Markov semigroup. In the
continuous setting this characterization is one of the
cornerstones of the theory initiated by Bakry and \'Emery \cite{BE85, BGL15}.

We shall present an application of our results to a particular
interacting particle system, namely the zero-range process on the
complete graph with constant rates. The best known entropic Ricci
bound for this model is 0. Using Theorem \ref{main_thm_mlsi} and
easily obtained diameter bound allows us to establish a new bound on
the mLSI constant for the zero range process.

\subsection*{Outline}
\label{sec:outline}

In Section \ref{sec:prelim}, we shall recall the definition and basic
results about the discrete transport distance $\cW$ and entropic Ricci
curvature bounds for Markov chains. In Section \ref{sec:grad-est}, we
shall give and equivalent characterization of entropic Ricci bounds in
terms of gradient estimates for the Markov semigroup. Section
\ref{sec:buser} will provide the proof of the discrete Buser theorem,
while Sections \ref{sec:poincare} and \ref{sec:mlsi} will be concerned
with the Poincar\'e and modified log Sobolev inequalities under joint
curvature and diameter bounds. Finally, in Section
\ref{sec:zero-range}, we consider applications to the zero range process.

\subsection*{Acknowledgments} M.F.~was supported by NSF FRG grant
DMS-1361122. M.E.~gratefully acknowledges support by the German
Research Foundation through the \emph{Hausdorff Center for Mathematics}. We
thank Michel Ledoux, Jan Maas, Emanuel Milman, Andr\'e Schlichting and
Prasad Tetali for discussions on this topic. This work was initiated
during a trimester on optimal transport organized at the \emph{Hausdorff
Institute for Mathematics} in Bonn, whose support is gratefully
acknowledged. We also benefited from the hospitality of the American Institute of Mathematics during the SQUARE meetings 
\textit{Displacement convexity for interacting Markov chains}. 

\section{Entropic Ricci curvature
  bounds for Markov chains}
\label{sec:prelim}

We briefly recall the definitions of the transport distance $\cW$ and
of entropic Ricci curvature bounds for finite Markov chains and some
of their consequences. For a detailed account we refer to the work of
Maas and Mielke \cite{Ma12, Mie12} where the discrete transport
distance and its associated Riemannian structure have been introduced
and to \cite{EM12} where entropic Ricci curvature bounds have been introduced
and studied.

\subsection{Transport distance and Ricci bounds}
\label{sec:def}

Let $\cX$ be a finite set let $Q:\cX\times\cX\to\R_+$ be a collection
of transition rates with the convention that $Q(x,x)=0$ for all
$x$. The operator $L$ acting on functions $\psi:\cX\to\R$ defined by 
\begin{align*}
  L\psi (x) = \sum\limits{y\in\cX}\big(\psi(y)-\psi(x)\big)Q(x,y)\;,
\end{align*}
is the generator of a continuous time Markov chain on $\cX$. We will
assume that $Q$ is irreducible, i.e.~for all $x,y\in\cX$ there exist
$(x_1=x,x_2,\dots,x_n=y)$ with $Q(x_i,x_{i+1})>0$. This implies the
existence of a unique stationary probability measure $\pi$ on
$\cX$. We will assume that $Q$ is reversible w.r.t.~$\pi$ in the sense
that the detailed balance condition $Q(x,y)\pi(x)=Q(y,x)\pi(y)$ holds
for all $x,y\in\cX$. We denoted by
\begin{align*}
 \PX := \Big\{ \, \rho : \cX \to \R_+ \ | \ 
     \sum_{x \in \cX} \pi(x) \rho(x)  = 1 \, \Big\}
\end{align*}
the set of \emph{probability densities} w.r.t.~$\pi$. Since the
measure $\pi$ is strictly positive and we can identify the set of
probability measures on $\cX$ with $\cP(\cX)$. The subset consisting
strictly positive probability densities is denoted by $\PXs$. We
consider the metric $\cW$ defined for $\rho_0, \rho_1 \in \PX$ by
\begin{align*}
 \cW(\rho_0, \rho_1)^2
   := \inf_{\rho, \psi} 
   \bigg\{  \frac12   \int_0^1 
  \sum_{x,y\in \cX} (\psi_t(x) - \psi_t(y))^2
    		 \hat\rho_t(x,y)  Q(x,y)\pi(x)
      \dd t 
          \bigg\}\;,
\end{align*}
where the infimum runs over all sufficiently regular curves $\rho :
[0,1] \to \PX$ and $\psi : [0,1] \to \R^\cX$ satisfying the
continuity equation
\begin{align} \label{eq:cont}
 \begin{cases}
 \displaystyle\ddt \rho_t(x) 
   + \displaystyle\sum_{y \in \cX} ( \psi_t(y) - \psi_t(x) ) \hat\rho_t(x,y) Q(x,y) ~=~0\qquad \forall x \in \cX\;, \\ 
  \rho|_{t=0} = \rho_0\;, \qquad \rho|_{t=1}  = \rho_1\;.
 \end{cases}
\end{align}
Here, given $\rho\in\PX$, we write
$\hat\rho(x,y):=\theta\big(\rho(x),\rho(y)\big)$ where
$\theta(s,t):=\int_0^1s^{1-p}t^p\dd p$ is the logarithmic mean of $s$
and $t$.

It has been shown in \cite{Ma12} that $\cW$ defines a distance on
$\cP(\cX)$ and that it is induced by a Riemannian structure on the
interior $\PXs$. The logarithmic mean serve the purpose to obtain a
discrete chain rule for the logarithm, namely
$\hat\rho(x,y)\big(\log\rho(x)-\log\rho(y)\big)=\rho(x)-\rho(y)$,
replacing the usual identity $\rho\nabla\log\rho=\nabla\rho$.
The distance $\cW$ is constructed in such a way that Markov semigroup $P_t=e^{tL}$ is the gradient flow of the entropy
\begin{align} \label{eq:entropy} \cH(\rho) = \sum_{x \in \cX} \pi(x)
  \rho(x) \log \rho(x)\;,
\end{align}
w.r.t.~the Riemannian structure induced by $\cW$, see
\cite{Ma12,Mie12}. It turns out that every pair of densities
$\rho_0, \rho_1 \in \PX$ can be joined by a constant speed geodesic,
i.e.~a curve $(\rho_s)_{s\in[0,1]}$ with
$\cW(\rho_s,\rho_t)=|s-t|\cW(\rho_0,\rho_1)$ for all $s,t\in[0,1]$. In
the spirit of the approach of Lott--Sturm--Villani \cite{LV, St06} the
following definition was given in \cite{EM12}.

\begin{definition}\label{def:intro-Ricci}
 $(\cX,Q,\pi)$ has \emph{ entropic Ricci curvature bounded from below by
    $\kappa \in \R$} if for any constant speed geodesic $\{\rho_t\}_{t
    \in [0,1]}$ in $(\PX, \cW)$ we have
  \begin{align}\label{eq:ent-convex}
    \cH(\rho_t) \leq (1-t) \cH(\rho_0) + t \cH(\rho_1) -
    \frac{K}{2} t(1-t) \cW(\rho_0, \rho_1)^2\;.
  \end{align}
  In this case, we write $\Ric(\cX,Q,\pi) \geq \kappa$.
\end{definition}

We should mention that several other notions of Ricci curvature for
Markov chains on discrete spaces have been proposed in the past few
years: Ollivier's coarse Ricci curvature \cite{Oll09}, convexity along
approximate geodesics \cite{BS}, discrete Bakry-\'Emery curvature
(defined in \cite{BE85}, and used in the discrete setting for example
in \cite{Sc99,LY10}), convexity along binomial interpolations \cite{GRST14}
or along entropic interpolations \cite{Leo13}.

\subsection{Riemannian structure and equivalent formulation}
\label{sec:Riem}

We will briefly describe the Riemannian structure induced by $\cW$ to
give an equivalent formulation of entropic Ricci bounds in terms of a
discrete analogue of Bochner's inequality.

For $\psi\in\R^\cX$ we denote by $\nabla\psi\in\R^{\cX\times\cX}$ the
discrete gradient $\nabla\psi(x,y)=\psi(y)-\psi(x)$. We denote by
$\cG=\{\nabla\psi:\psi\in\R^\cX, \psi(x_0)=0\}$ the set of discrete
gradients modulo constants, for some fixed $x_0\in\cX$. In
\cite{Ma12} it has been shown that for each $\rho\in\PXs$ the map
$$\nabla\psi\mapsto\sum_y\nabla\psi(x,y)Q(x,y)$$
is a linear bijection between $\cG$ and the tangent space
$\mathcal T=\{s\in\R^\cX:\sum_x s(x)\pi(x)=0\}$ to $\PXs$ at $\rho$.
A Riemannian tensor on $\PXs$ can be defined using this identification
by introducing the scalar product $\ip{\cdot,\cdot}_\rho$ given by
\begin{align*}
  \ip{\nabla\phi,\nabla\phi}_\rho = \frac12\sum_{x,y}\nabla\psi(x,y)\nabla\phi(x,y)\hat\rho(x,y)Q(x,y)\pi(x)\;.
\end{align*}
Then $\cW$ is the associated Riemannian distance. We will set
$\cA(\rho,\psi):=|\nabla\psi|^2_\rho$.  Convexity of the entropy along
$\cW$-geodesics is controlled by lower bounds on the Hessian of the
entropy $\cH$ in the Riemannian structure just defined. An explicit
calculation of the Hessian at $\rho\in\PXs$ yields
\begin{align*}
  {\rm Hess}\cH(\rho)[\nabla\psi]=\frac12\sum_{x,y}\left[\frac12 \hat L\rho(x,y)|\nabla\psi(x,y)|^2-\hat\rho(x,y)\nabla\psi(x,y)\nabla L\psi(x,y)\right]Q(x,y)\pi(x)\;,
\end{align*}
where we set
$\hat
L\rho(x,y)=\partial_1\theta\big(\rho(x),\rho(y)\big)L\rho(x)+\partial_2\theta\big(\rho(x),\rho(y)\big)L\rho(y)$.
Setting for brevity $\cB(\rho,\psi):={\rm Hess}\cH(\rho)[\nabla\psi]$
we have the following reformulation of entropic Ricci bounds. Note
that the statement is non-trivial since the Riemannian structure
degenerates at the boundary of $\PX$.

\begin{proposition}\label{thm:Ric-equiv}
  We have $\Ric(\cX,Q,\pi)\geq\kappa$ if and only if for every
  $\rho\in\PXs$ and $\psi\in\R^\cX$ we have
\begin{align*}
 \cB(\rho, \psi) \geq \kappa \cA(\rho, \psi)\;.
\end{align*}
\end{proposition}

\subsection{Functional inequalities}
\label{sec:fi}

Here we introduce the discrete functional inequalities that we shall
be interested in. The discrete Dirichlet form $\cE$
associated to the Markov triple $(\cX,Q,\pi)$ is given by
\begin{align*}
  \cE(\phi,\psi)=\frac12\sum_{x,y}\nabla\phi(x,y)\nabla\psi(x,y)Q(x,y)\pi(x)\;.
\end{align*}
We say that $(\cX,Q,\pi)$ satisfies for a constant $\lambda>0$
\begin{itemize}
\item a Poincar\'e inequality P$(\lambda)$ if
  \begin{align}\label{eq:PI}
\Var_{\pi}(\psi) \leq \frac{1}{\lambda}\cE(\psi,\psi) \quad\forall \psi\in\R^\cX\;,
  \end{align}
\item a modified logarithmic Sobolev inequality MLSI$(\lambda)$ if 
  \begin{align}\label{eq:MLSI}
    \cH(\rho)\leq \frac{1}{2\lambda}\cE(\rho,\log\rho) \quad\forall \rho\in\PXs\;,
  \end{align}
\item a modified Talagrand inequality T$_\cW(\lambda)$ if 
  \begin{align}\label{eq:Talagrand}
    \cW(\rho,\one)^2\leq \frac2\lambda\cH(\rho)\quad\forall \rho\in\PX\;.
  \end{align}
\end{itemize}
Here $\Var_\pi(\psi):=\pi[\psi^2]-\pi[\psi]^2$. The Poincar\'e
inequality is a spectral estimate: the optimal constant coincides with
the the smallest nonzero eigenvalue of the generator $L$ of the Markov
chain. It is well known that these functional inequalities govern the
trend to equilibrium for the Markov semigroup $P_t=e^{tL}$. Indeed,
noting that
\begin{align*}
  \ddt\cH(P_t\rho)=-\cE(P_t\rho,\log P_t\rho)\;,\quad \ddt \Var_\pi(P_t\psi)=-\cE(P_t\psi,P_t\psi)\;,
\end{align*}
we see that \eqref{eq:PI} and \eqref{eq:MLSI} are respectively
equivalent to the exponential convergence estimates
\begin{align}\label{eq:exp-conv}
  \cH(P_t\rho)\leq e^{-2\lambda t}\cH(\rho)\;,\quad \Var_\pi(P_t\psi)\leq e^{-\lambda t}\Var_\pi(\psi)\;.
\end{align}
We shall often denote the entropy production functional by
$\cI(\rho):=\cE(\rho,\log\rho)$. The modified Talagrand inequality
implies concentration properties for the measure $\pi$.

It has been shown in \cite{EM12} that entropic Ricci bounds are
intimately related with the functional inequalities above.
The bound $\Ric(\cX,Q,\pi)\geq\kappa$ for $\kappa\in\R$ implies the
$\cH\cW \cI(\kappa)$ inequality
\begin{align}\label{eq:HWI}
  \cH(\rho_1)-\cH(\rho_0)\leq \cW(\rho_0,\rho_1)\sqrt{\cI(\rho_1)} -\frac{\kappa}2\cW(\rho_0,\rho_1)^2\;.
\end{align}
If $\kappa>0$, this inequality readily implies MLSI$(\kappa)$. The
latter in turn was shown to imply the modifies Talagrand inequality
T$_\cW(\kappa)$ in analogy with the result of Otto--Villani in the
continuous case. By a linearization argument, it can be shown that
MLSI$(\kappa)$ and T$_\cW(\kappa)$ both imply the Poincar\'e
inequality P$(\kappa)$. The converse is not true in general, see for
example \cite{BGL15}.

\subsection{Distances on $\cX$}
\label{sec:dist}

The transport distance $\cW$ on $\PX$ gives rise to a distance $d_\cW$
on $\cX$ by restriction to Dirac masses. More precisely, we set for
$x,y\in\cX$
\begin{align*}
  d_\cW(x,y) = \cW(\delta_x,\delta_y)\;.
\end{align*}

We can give an upper bound on $d_\cW$ in terms of a weighted graph distance.
Let us define a distance $d_Q$ on $\cW$ by setting
\begin{align*}
  d_Q(x,y) =\inf\Big\{\sum_{i=0}^{n-1}\frac{1}{\sqrt{\min\big(Q(z_i,z_{i+1}),Q(z_{i+1},z_i)\big)}} \Big\}\;,
\end{align*}
where the infimum runs over all sequences $z_0=x, z_1,\dots,z_n=y$
such that $Q(z_i,z_{i+1})>0$ (and hence also $Q(z_{i+1},z_i)>0$ by
detailed balance).

\begin{lemma}\label{lem:distance-comp}
  For any $x,y\in\cX$ we have
  \begin{align*}
    d_\cW(x,y)\leq c d_Q(x,y)\;,\quad  c=\int_0^1\frac{\dd r}{\sqrt{2\theta(1-r,1+r)}}\approx 1.56\;.
  \end{align*}
\end{lemma}

This is a reinforcement of \cite[Lem.~2.13]{EM12}, where $d_\cW$ has
been compared to $d_g/\sqrt{Q_*}$, where $d_g$ is the unweighted graph
distance obtained by replacing the summands in the definition of $d_Q$
by $1$ and $Q_*=\min\{Q(x,y):Q(x,y)>0\}$ is the minimal transition
rate. 

\begin{proof}
  We argue similarly as in \cite[Lem.~2.13]{EM12}. We shall use that
  for $x,y\in\cX$ with $Q(x,y)>0$ the distance
  $\cW(\delta_x,\delta_y)$ can be estimated by the distance on a two
  point space. More precisely, \cite[Thm.~2.4, Lem.~3.14]{Ma12} and
  their proofs yield the estimate
  \begin{align*}
    \cW\left(\frac{\one_{\{x\}}}{\pi(x)},\frac{\one_{\{y\}}}{\pi(x)}\right)\leq c\sqrt{\frac{\max\big(\pi(x),\pi(y)\big)}{Q(x,y)\pi(x)}}
       = c\sqrt{\frac{1}{\min\big(Q(x,y),Q(y,x)\big)}}\;,
  \end{align*}
  where we have used detailed balance in the last step. The result
  than follows by the triangle inequality for $\cW$ and by taking the
  infimum over all sequences connecting $x$ to $y$.
\end{proof}

\subsection{Notation}
\label{sec:notation}

In order to alleviate notation in the sequel we introduce the following concepts. Given two functions $\phi,\psi\in\R^\cX$ we denote their scalar product in $L^2(\pi)$ by
\begin{align*}
  \ip{\phi,\psi}_\pi=\sum_{x\in\cX}\phi(x)\psi(x)\pi(x)\;.
\end{align*}
For two functions $\Phi,\Psi\in\R^{\cX\times\cX}$ defined on edges we define
\begin{align*}
  \ip{\Phi,\Psi}_\pi=\frac12\sum_{x,y\in\cX}\Phi(x,y)\Psi(x,y)Q(x,y)\pi(x)\;.
\end{align*}
As a consequence of the detailed balance assumption, the generator $L$
is selfadjoint and we
have an integration by parts formula
\begin{align*}
  \ip{\psi,L\phi}_\pi=-\ip{\nabla \phi,\nabla \psi}_\pi=\ip{L\psi,\phi}_\pi\;.
\end{align*}
We note moreover, that that the Riemannian metric tensor associate to
$\cW$ can be rewritten as
$\cA(\rho,\psi)=\ip{\nabla\psi,\nabla\psi}_\rho=\ip{\hat\rho\cdot\nabla\psi,\nabla\psi}_\pi$,
where the product $\hat\rho\cdot\nabla\psi$ is defined
component-wise. Similarly, the Hessian of the entropy can be rewritten
in the compact form
\begin{align*}
\cB(\rho,\psi)=\frac12\ip{\hat L\rho\cdot\nabla\psi,\nabla\psi}_\pi-\ip{\nabla\psi,\hat\rho,\nabla L\psi}_\pi\;.
\end{align*}
We also define the $\Gamma$-operator $\Gamma:\R^\cX\times\R^\cX\to\R^\cX$ by setting
\begin{align*}
  \Gamma(\phi,\psi)(x)=\sum_y\nabla\phi(x,y)\nabla\psi(x,y)Q(x,y)\;,
\end{align*}
and set $\Gamma(\phi)=\Gamma(\phi,\phi)$.

\section{Gradient estimates}
\label{sec:grad-est}

In this section we show that entropic Ricci curvature lower bounds are
equivalent to certain gradient estimates for the Markov semigroup
$P_t=\e^{tL}$. These will be crucial in establishing functional
inequalities in the sequel.

\begin{theorem}[Gradient estimate]\label{thm:BLKN}
  A Markov triple $(\cX,Q,\pi)$ satisfies the entropic Ricci bound
  $\Ric(\cX,Q,\pi)\geq \kappa$ if and only if for every $\psi\in\R^\cX$ and
  $\rho\in\PX$ we have
  \begin{align}\label{eq:BLK}
    \abs{\nabla P_t\psi}^2_\rho  ~\leq~ \e^{-2\kappa t} \abs{\nabla \psi}^2_{P_t\rho}\;.
 \end{align}
\end{theorem}

\begin{remark}
  More explicitly, the gradient estimate reads as follows:
  \begin{align*}
  \frac12\sum_{x,y}\abs{\nabla P_t \psi}^2(x,y)\hrho(x,y)Q(x,y)\pi(x)\leq
   \e^{-2\kappa t} \frac12\sum_{x,y}\abs{\nabla \psi}^2(x,y)\widehat{P_t\rho}(x,y)Q(x,y)\pi(x)\;.
  \end{align*}
\end{remark}

The proof follows the standard semigroup interpolation argument,
slightly adapted to our setting.

\begin{proof}
  First we note that if \eqref{eq:BLK} holds for all $\psi$ and
  $\rho\in\PXs$ then it also holds for all $\rho\in\PX$.  Now fix
  $\psi\in\R^\cX$, $\rho\in\PXs$ and $t>0$. We define for $s\in[0,t]$:
  \begin{align*}
    \Phi(s) ~=~ \e^{-2\kappa s}\abs{\nabla P_{t-s} \psi}^2_{P_s\rho} ~=~ \e^{-2\kappa s}\ip{\nabla P_{t-s}\psi, \widehat{P_s\rho}\cdot\nabla P_{t-s}\psi}_\pi\;.
  \end{align*}
  Note that $\Phi(0)=\abs{\nabla P_{t} \psi}^2_{\rho}$ and
  $\Phi(t)=\e^{-2\kappa t}\abs{\nabla \psi}^2_{P_t\rho}$. We immediately compute the
  derivative of $\Phi$. Putting $\psi_s=P_s\psi$ and $\rho_s=P_s\rho$ we get:
  \begin{align*}
    \Phi'(s) ~&=~ \e^{-2\kappa s}\Big[\ip{\nabla \psi_{t-s},\hat{L} \rho_s\cdot\nabla \psi_{t-s}}_\pi
                    -2 \ip{\nabla \psi_{t-s},\hat{\rho_s}\cdot\nabla L \psi_{t-s}}_\pi
                    -2\kappa \ip{\nabla \psi_{t-s},\hat{\rho_s}\nabla \psi_{t-s}}_\pi\Big]\\
             ~&=~ 2\e^{-2\kappa s} \Big[\cB(\rho_s,\psi_{t-s}) - \kappa \cA(\rho_s,\psi_{t-s})\Big]\;.
  \end{align*}
  If $\Ric(\cX,Q,\pi)\geq \kappa$ holds, we conclude from Proposition
  \ref{thm:Ric-equiv} that $\Phi'(s)\geq 0$ and obtain \eqref{eq:BLK}.
  For the converse implication we derivate \eqref{eq:BLK} at
  $t=0$. More precisely, assume that \eqref{eq:BLK} holds. Then we
  have that
  \begin{align*}
    0 &\leq e^{-2\kappa t}\abs{\psi}^2_{P_t\rho}-\abs{P_t\psi}^2_\rho\\
      &=\big(e^{-2\kappa t}-1\big)\abs{\psi}^2_\rho + e^{-2\kappa t}\big(\abs{\psi}^2_{P_t\rho}-\abs{\psi}^2_\rho\big)
       -\abs{P_t\psi}^2_{\rho} + \abs{\psi}^2_\rho\;.
  \end{align*}
  Dividing by $t$ and letting $t\to0$ we obtain that
  $\cB(\rho,\psi)-\kappa\cA(\rho,\psi)\geq 0$ which yields the claim
  again by Proposition \ref{thm:Ric-equiv}.
\end{proof}

\begin{remark}
  The previous result is in close analogy with the classical
  Bakry--\'Emery gradient estimate for the heat semigroup on a
  Riemannian manifold $M$ with $\Ric\geq \kappa$ which states that for
  any smooth function $\psi$ it holds
  $|\nabla P_t\psi|^2\leq e^{-2\kappa t} P_t|\nabla
  \psi|^2$. Integrating this estimate against a function $\rho$ yields
  \begin{align*}
    \int_M|\nabla P_t\psi|^2\rho \leq e^{-2\kappa t} \int_M|\nabla\psi|^2P_t\rho\;,
  \end{align*}
  which closely resembles \eqref{eq:BLK} except for the appearance of
  the logarithmic mean.
\end{remark}

Using the freedom in the choice of $\rho$ in the gradient estimate
\eqref{eq:BLK}, we can deduce a more explicit estimate, that does not involve
logarithmic means. To this end we introduce the heat kernel associated
to the continuous-time Markov chain. We put
$p_t(x,y)=\pi(y)^{-1}P_t\one_{y}(x)$. From the symmetry and linearity
of $P_t$ it is immediate to check that
$P_tf(x)=\sum_y p_t(x,y)f(y)\pi(y)$ and $p_t(x,y)=p_t(y,x)$.

\begin{corollary}\label{cor:BLKN}
  If $\Ric(\cX,Q,\pi)\geq \kappa$ holds, we have for any
  $\psi\in\R^\cX$ and all $x,y\in\cX$:
  \begin{align}\label{eq:BLKN-ptws1}
    \frac12\abs{\nabla P_t\psi}^2(x,y)Q(x,y)\pi(x) ~&\leq~ \e^{-2\kappa t} \big[P_t\Gamma(\psi)(x)\pi(x) + P_t\Gamma(\psi)(y)\pi(y)\Big]\;.
  \end{align}
\end{corollary}

\begin{proof}
  Starting from \eqref{eq:BLK} we choose
  $\rho=\one_{x}+\one_{y}$. Note that it is not a probability
  density, but this does no harm, since both sides of \eqref{eq:BLK}
  are homogeneous in $\rho$. Note that
  $P_t\rho(u)= p_t(x,u)\pi(x)+p_t(y,u)\pi(y)$. Using the estimate
  $\theta(s,t)\leq (s+t)/2$ and using symmetry we obtain
  \begin{align*}
    \frac12\abs{\nabla P_t\psi}^2(x,y)Q(x,y)\pi(x)
    ~&\leq~
     \e^{-2\kappa t}\frac12\sum_{u,v}\abs{\nabla \psi}^2(u,v)\big[p_t(x,u)\pi(x)+p_t(y,u)\pi(y)\big]Q(u,v)\pi(u)\\
    &=~
    \e^{-2\kappa t} \big[P_t\Gamma(\psi)(x)\pi(x) + P_t\Gamma(\psi)(y)\pi(y)\Big]\;.
  \end{align*}
\end{proof}

Next, we derive a reverse Poincar\'e inequality along the Markov
semigroup.

\begin{theorem}[Reverse Poincar\'e inequality]\label{thm:reverse-Poincare}
  Assume that $\Ric(\cX,Q,\pi)\geq \kappa$. Then we have for any
  $\rho\in\cP(\cX)$ and $\psi\in\R^{\cX}$:
  \begin{align}\label{eq:reverse-Poincare}
    \ip{\psi^2,P_t\rho}_\pi - \ip{(P_t\psi)^2,\rho}_\pi~\geq~\frac{\e^{2\kappa t}-1}{\kappa}\abs{\nabla P_t\psi}^2_\rho\;.
  \end{align}
\end{theorem}

\begin{proof}
  We put $\Phi(s)=\ip{(P_{t-s}\psi)^2,P_s\rho}_\pi$ and calculate,
  putting $g=P_{t-s}\psi$ and $h=P_s\rho$,
  \begin{align*}
    \Phi'(s) ~&=~ \ip{\Delta(P_{t-s}\psi)^2 - 2P_{t-s}\psi\Delta P_{t-s}\psi,P_s\rho}_\pi\\
              &=~ \sum_{x,y}h(x)\big[g(y)^2-g(x)^2\big]Q(x,y)\pi(x) -2\sum_{x,y}h(x)g(x)\big[g(y)-g(x)\big]Q(x,y)\pi(x)\\
              &=~ \sum_{x,y}h(x)\big[g(y)-g(x)\big]^2Q(x,y)\pi(x)\\
              &=~ \sum_{x,y}\frac{h(x)+h(y)}{2}\big[g(y)-g(x)\big]^2Q(x,y)\pi(x)\\
              &\geq~ \sum_{x,y}\theta(h(x),h(y))\big[g(y)-g(x)\big]^2Q(x,y)\pi(x)\\
              &=~ 2\abs{\nabla P_{t-s}\psi}^2_{P_s\rho}\;.
  \end{align*}
  Using the gradient estimate \eqref{eq:BLK} we conclude
  $\Phi'(s)\geq\e^{2\kappa s}\abs{\nabla P_t\psi}^2_\rho$ which yields
  the claim.
\end{proof}

From the previous theorem we can derive a uniform bound on the
gradient. Fix $x,y\in\cX$ and put
$\rho=\frac1{2\pi(x)}\delta_x+\frac1{2\pi(y)}\delta_y$. Then
\eqref{eq:reverse-Poincare} yields
\begin{align*}
  \frac{\e^{2\kappa t}-1}{\kappa}\abs{\nabla P_t\psi}^2(x,y)\theta\big(Q(x,y),Q(y,x)\big)~\leq~\norm{\psi}^2_\infty\;.
\end{align*}
Putting $Q_{*}=\min\{Q(x,y): x,y~s.t.~Q(x,y)>0\}$ we obtain
\begin{align}\label{eq:sup-reverse-Poincare}
  \frac{\e^{2\kappa t}-1}{\kappa}\max\limits_{x,y:Q(x,y)>0}\abs{\nabla P_t\psi}(x,y)~\leq~\frac{1}{\sqrt{Q_{*}}}\norm{\psi}_\infty\;.
\end{align}

As a consequence we obtain the following $L^1$ bound for the semigroup.

\begin{lemma}
  If $\Ric(\cX,Q,\pi)\geq\kappa$ we have for any $\psi\in\R^{\cX}$ and $t\leq 1/2\abs{\kappa}$:
  \begin{align}\label{eq:L1-L1}
    \norm{\psi-P_t\psi}_{L^1(\pi)}~\leq~\frac{2\sqrt{t}}{\sqrt{Q_{*}}}\norm{\nabla \psi}_{L_1}\:,
  \end{align}
  or more explicitly
  \begin{align*}
    \sum_x\abs{\psi(x)-P_t\psi(x)}\pi(x)~\leq~\frac{2\sqrt{t}}{\sqrt{Q_{*}}}\sum_{x,y}\abs{\nabla \psi}(x,y)Q(x,y)\pi(x)\;.
  \end{align*}
\end{lemma}

\begin{proof}
  Fix a function $g$ with $\abs{g}\leq 1$. Then we estimate
  \begin{align*}
    \ip{g,\psi-P_t\psi}_{\pi} ~&=~-\int_0^t\ip{g,\Delta P_s \psi}_\pi\dd s\\
                        &=~ \int_0^t\ip{\nabla P_s g,\nabla \psi}_\pi\dd s\\
                        &\leq~ \norm{\nabla \psi}_{L_1}\frac{1}{\sqrt{Q_{*}}}\int_0^t\frac{1}{\sqrt{s}} \dd s\\
                        &=~ \norm{\nabla \psi}_{L_1}\frac{2\sqrt{t}}{\sqrt{Q_{*}}}\;.
  \end{align*}
  Here we have used \eqref{eq:sup-reverse-Poincare} and the fact that
  $\big(\e^{2\kappa t}-1\big)/\kappa \geq t$ for $0<t\leq
  \frac1{2\abs{\kappa}}$. Taking the supremum over $g$ yields the claim.
\end{proof}

Finally, we derive an exponential decay estimate for the $\Gamma$
operator.
\begin{proposition} \label{prop_erbar} Assuming that
  $\Ric(\cX,Q,\pi)\geq\kappa\in\R$, we have for any $f\in\R^\cX$
  \begin{align*}
\pi\left[\Gamma(P_tf)\right] \leq e^{-2\kappa t}\pi\left[ \Gamma(f)\right] \;.
  \end{align*}
  Moreover, $P_tf$ is Lipschitz with constant
  $\norm{f}_{\infty}\sqrt{\kappa/(e^{2\kappa t}-1)}$ for the
  distance $d_{\mathcal{W}}$.
\end{proposition}

\begin{proof}
  The first part is a direct consequence of the gradient estimate
  \eqref{eq:BLK} when taking $\rho \equiv 1$. For the second part, fix
  $\eps>0$ and consider a pair $(\rho_s,\psi_s)_{s\in[0,1]}$
  satisfying \eqref{eq:cont} such that
  $\cW(\delta_x,\delta_y)^2\leq \int_0^1\abs{\psi_s}^2_{\rho_s}\dd s
  +\eps$.
  Then we can estimate using Jensen's inequality and the reverse
  Poincar\'e inequality \ref{thm:reverse-Poincare}:
  \begin{align*}
    P_tf(x)-P_tf(y) &= \int_0^1\ip{\nabla P_tf,\nabla \psi_s}_{\rho_s} \dd s
              \leq \int _0^1\abs{\nabla P_tf}_{\rho_s}\abs{\nabla \psi_s}_{\rho_s} \dd s\\
              &\leq \norm{f}_{\infty}\sqrt{\kappa/(e^{2\kappa t}-1)}\int _0^1\abs{\nabla \psi_s}_{\rho_s} \dd s\\
               &\leq \norm{f}_{\infty}\sqrt{\kappa/(e^{2\kappa t}-1)} \sqrt{\cW(\delta_x,\delta_y)^2+\eps}\;.
  \end{align*}
  Letting $\eps\to0$ yields the claim.
\end{proof}

\section{Isoperimetric estimate and Buser inequality}
\label{sec:buser}

Now we can formulate an isoperimetric estimate which will immediately
yield the discrete Buser inequality, see Theorem \ref{main_thm_buser}. To this end we recall that the
perimeter measure $\pi^+$ of a subset $A\subset\cX$ is defined by
\begin{align*}
  \pi^+(\partial A)~=~\sum_{x\in A, y\in A^c} Q(x,y)\pi(x)\;.
\end{align*}

\begin{theorem}\label{thm:isoperimetric}
  Assume that $\Ric(\cX,Q,\pi)\geq\kappa$ holds and let $\lambda_1$ be
  the spectral gap of $L$. Then for any subset $A\subset\cX$ we have
  \begin{align}\label{eq:isoperimetric}
    \pi^+(\partial A)~\geq~\frac{1}{3}\sqrt{Q_{*}}\min\Big(\frac{\lambda_1}{\sqrt{\kappa}},\sqrt{\lambda_1}\Big) \pi(A)\big(1-\pi(A)\big)\;.
  \end{align}
\end{theorem}

\begin{proof}
  We apply \eqref{eq:L1-L1} to the indicator function $\chi_A$ and
  obtain
  \begin{align*}
    \frac{2\sqrt{t}}{\sqrt{Q_{*}}}\cdot\pi^+(\partial A) ~&=~ \frac{2\sqrt{t}}{\sqrt{Q_{*}}} \norm{\nabla \chi_A}_{L^1} ~\geq~ \norm{\chi_A-P_t\chi_A}_{L^1(\pi)}\\
                                &=~ 2\pi(A)-2\norm{P_{\frac t2}\chi_A}_{L^2(\pi)}^2\\
                                &=~ 2\Big[\pi(A)-\pi(A)^2-\norm{P_{\frac t2}(\chi_A-\pi(A))}_{L^2(\pi)}^2\Big]\\
                                &\geq~ 2\Big[\pi(A)-\pi(A)^2-\e^{-\lambda_1 t}\norm{\chi_A-\pi(A)}_{L^2(\pi)}^2\Big]\\
                                &=~ 2\pi(A)\big(1-\pi(A)\big)\big(1-\e^{-\lambda_1 t}\big)\;.
  \end{align*}
  Now we conclude by optimizing in $t$. That is, if
  $\lambda_1\geq2\abs{\kappa}$ we can take $t=\frac1{\lambda_1}$ and in the
  opposite case we take $t=\frac1{2\kappa}$.
\end{proof}

We recall that the Cheeger constant $h$ of the Markov chain
$(\cX,Q,\pi)$ is defined as the optimal constant in the previous isoperimetric estimate, i.e.
\begin{align*}
  h = \max_{A\subset \cX}\frac{\pi^+(\partial A)}{\pi(A)\big(1-\pi(A)\big)}\;.
\end{align*}

\begin{corollary}[Buser inequality]\label{cor:Buser}
  If $\Ric(\cX,Q,\pi)\geq \kappa$, we have the following Buser
  inequality:
  \begin{align}\label{eq:Buser}
    h~\geq~ \frac{1}{3}\sqrt{Q_{*}}\min\Big(\frac{\lambda_1}{\sqrt{\kappa}},\sqrt{\lambda_1}\Big)\;.
  \end{align}
  In particular, if $\kappa\geq0$, we have $h\geq\frac{\sqrt{\lambda_1}}{3}\sqrt{Q_{*}}$.
\end{corollary}

\begin{proof}
  This follows immediately from the previous proposition and the
  definition of the Cheeger constant $h$. The second statement follows
  by recalling that if $\kappa>0$, $\Ric(\cX,Q,\pi)\geq\kappa$ implies
  that $\lambda_1\geq \kappa$.
\end{proof}

\begin{corollary}
When $\Ric(\cX,Q,\pi)\geq\kappa$, the Poincar\'e inequality P$(\lambda_1)$ implies an $L^1$ Poincar\'e inequality
\begin{align*}
\sum_x |\psi(x) - \pi[\psi]|\pi(x) \leq \frac{4}{c(\kappa,\lambda_1)}\sum |\nabla \psi|Q(x,y)\pi(x) \quad\forall \psi\in\R^\cX\;.
\end{align*}
where $c(\kappa,\lambda_1)$ is the constant in the right hand side of \eqref{eq:Buser}. 
\end{corollary}

The equivalence of the Cheeger isoperimetric inequality and the $L^1$
Poincar\'e inequality is a well-established fact. We briefly prove
that the Cheeger inequality \eqref{eq:isoperimetric} implies the $L^1$
Poincar\'e inequality for the sake of completeness. The reverse
implication is immediately established by applying the $L^1$
Poincar\'e inequality to indicators of sets.

\begin{proof}
  First, we shall establish the inequality when $\pi[\psi]$ is
  replaced by a median of $\psi$. Let $\psi$ be a function with median
  $0$. Writing $\chi_A$ for the indicator of a set $A$, we have
\begin{align*}
\sum_{x,y} |\psi(x) - \psi(y)|&Q(x,y)\pi(x) = \int_{-\infty}^{+\infty}{\underset{x,y}{\sum} \hspace{1mm} \chi_{[\psi(y),\psi(x)]}(t)Q(x,y)dt} \\
&= \int_{-\infty}^{+\infty}{\pi^+(\psi>t)dt} \\
&\geq \int_{-\infty}^{+\infty}{c(\kappa,\lambda_1)\pi(\psi>t)\pi(\psi<t)dt} \\
&\geq c(\kappa,\lambda_1)\int_{-\infty}^0{\frac{\pi(\psi<t)}{2}dt} + c(\kappa,\lambda_1)\int_{0}^{+\infty}{\frac{\pi(\psi>t)}{2}dt} \\
&= \frac{c(\kappa,\lambda_1)}{2}\sum_x |\psi(x)|\pi(x)\;. 
\end{align*}

We can then replace the median by the mean since $\pi\big[|\psi -\pi[\psi]|\big] \leq 2\pi\big[|\psi|\big]$.
\end{proof}

\section{Poincar\'e inequality}
\label{sec:poincare}

The aim in this section is to prove that if the curvature of a Markov
is non-negative, then the spectral gap is controlled by the
diameter. This is a discrete analog of the following classical result
of Li and Yau \cite{LY}:
\begin{theorem}
  Let $M$ be a manifold with non-negative curvature and finite
  diameter $D$. Then the spectral gap of the manifold satisfies
$$\lambda_1 \geq \frac{\pi^2}{D^2}.$$
\end{theorem}
In the continuous setting, the constant $\pi^2$ is known to be
optimal. Moreover, this result is rigid, in the following sense: if
equality holds, then the manifold is isometric to the one-dimensional
torus of diameter $D$ \cite{HW07}. As we shall discuss in Section
\ref{sect_milman}, results of Milman \cite{Mil09, Mil09b, Mil10}
show that on geodesic spaces the spectral gap can be controlled even
by measure concentration properties.

The proof is inspired by \cite{GRS11}, although the proof of the weak
Poincar\'e inequality follows a different line of arguments, inspired
by \cite{Le11}. First we establish a weak Poincar\'e inequality, under
the assumption that curvature is non-negative and that the invariant
measures has the exponential concentration property. Then we establish
a tight Poincar\'e inequality under the stronger assumption of bounded
diameter.

\subsection{Weak Poincar\'e inequality under exponential concentration}

\begin{definition}
  A probability measure on a metric space is said to satisfy a
  concentration property with profile
  $\beta : \R_+ \longrightarrow [0,1]$ if, for any set $A$ such that
  $\mu(A) \geq 1/2$, we have
  \begin{align}\label{eq:concentration-def}
\mu(A_r^c) \leq \beta(r); \hspace{5mm} A_r^c := \{x; d(x,A) > r\}\;.
  \end{align}
  In particular, we shall say that a probability measure satisfies the
  exponential concentration property with constants $M$ and $\alpha$
  if \eqref{eq:concentration-def} holds with
  $\beta(r) := Me^{-\alpha r}$. Similarly, it is said to satisfy a
  Gaussian concentration property with constants $M$ and $\rho$ if it
  admits $\beta(r) = Me^{-\rho r^2}$ as a concentration profile.
\end{definition}

A concentration profile governs the tail behavior of the
measure. Heuristically, the exponential (resp. Gaussian) concentration
property compares the behavior at infinity with that of the
exponential (resp. Gaussian) measure. When the diameter is bounded by
$D$, it is easy to see that an exponential (resp. Gaussian)
concentration property holds with constant $\alpha=1/D$ and $M=e$
(resp. $\rho=1/D^2$ and $M=e$). However, this estimate is often much
worse than the optimal concentration estimate, and concentration can
hold for unbounded spaces (for example, $\R^d$ equipped with a
Gaussian measure).

A classical result in the study of concentration of measure is that
one can use functional inequalities to establish concentration with
some profile. In particular, the Poincar\'e inequality implies an
exponential concentration property, while a (modified) logarithmic
Sobolev inequality implies Gaussian concentration
\cite{Led_book}. Moreover, one can show that Gaussian concentration is
equivalent to a transport-entropy inequality for the Wasserstein
distance $W_1$ \cite{BG}. In the Riemannian setting, Milman
\cite{Mil09, Mil09b, Mil10} showed that the converse is true when
curvature is non-negative: functional inequalities and concentration
properties are then actually \emph{equivalent}. We shall discuss this
aspect further in Section \ref{sect_milman}.

\begin{proposition} \label{prop_weak_poincare} Let $(\cX,Q,\pi)$ be a
  Markov chain with $\Ric(\cX,Q,\pi)\geq\kappa\in\R$ and assume that
  $\pi$ has exponential concentration with respect to the distance
  $d_{\mathcal{W}}$ with constants $\alpha$ and $M$. Then for any
  $t > 0$ and $f\in\R^\cX$ we have
\begin{align*}
\Var_{\pi}(f) \leq \frac{1 - e^{-2\kappa t}}{\kappa}\pi\big[\Gamma(f)\big] + \frac{\kappa}{e^{2\kappa t}-1}\frac{2\norm{f}_{\infty}^2}{\alpha^2}\;.
\end{align*}
In particular, if $\Ric(\cX,Q,\pi)\geq0$, then
\begin{align*}
\Var_{\pi}(f) \leq 2t\pi\left[\Gamma(f)\right] + \frac{M\norm{f}_{\infty}^2}{\alpha^2 t}\;.
\end{align*}
If moreover the diameter of $(\cX,d_\cW)$ is bounded by $D$, we have
\begin{align*}
\Var_{\pi}(f) \leq 2t\pi\left[\Gamma(f)\right] + \frac{D^2\norm{f}_{\infty}^2}{4 t}\;.
\end{align*}
\end{proposition}

Note that by comparing $d_{\mathcal{W}}$ with the weighted graph
distance, we can use concentration with respect to the graph distance
instead. This will worsen the constant by a universal factor, see
Lemma \ref{lem:distance-comp}. 

\begin{proof}
  First, we have by Proposition \ref{prop_erbar}
\begin{align*}
\Var_{\pi}(f)&= \Var_{\pi}(P_tf) + 2\int_0^t{\pi\left[\Gamma(P_tf)\right] dt} \\
& \leq \Var_{\pi}(P_tf) + \frac{1 - e^{-2\kappa t}}{\kappa}\pi\left[\Gamma(f)\right]\;.
\end{align*}
Since $\pi$ satisfies the exponential concentration property with
constant $\alpha$ and $M$, for any Lipschitz function $f$, we have
\begin{align}\label{eq:var-bound}
\Var_{\pi}(f) \leq \frac{2M}{\alpha^2}||f||_{lip}^2
\end{align}
and applying this to $P_tf$ using Proposition \ref{prop_erbar} yields
the result. The variance bound \eqref{eq:var-bound} is obtained just
by integrating the concentration bound, and using the fact that
$\Var_{\pi}(f) \leq \pi[(f-m)^2]$, where $m$ is a median of $f$.
The estimate for the case when the diameter is bounded by $D$ is
obtained by using the estimate
$\Var_{\pi}(f) \leq \frac{D^2}{2}||f||_{lip}^2$.
\end{proof} 

\subsection{Tight Poincar\'e inequality under a diameter bound}

The proof of Theorem \ref{main_thm_sg} follows the ideas of
\cite{GRS11} and relies on a discrete analogue of the HWI inequality
that is given by $\cH\cW\cI(\kappa)$ in \eqref{eq:HWI}.

One of the obstacles to applying the strategy of \cite{GRS11} is that
the quantity $\cI(\rho)$ that appears in the $\cH\cW\cI(\kappa)$
inequality is given by
\begin{align*}
  \cI(\rho)=\cE(\rho,\log\rho)=\pi\left[\Gamma(\rho,\log \rho)\right]
\end{align*}
and that in the discrete setting this is the latter expression is
different from the term $\pi\left[\Gamma(\sqrt{\rho})\right]$ which
naturally appears in the Poincar\'e inequality. To this end we need
the following comparison result.

\begin{lemma}\label{lem:poinc-comp}
For any $\alpha>0$ the following are equivalent: 
\begin{align*}
(i)& \quad \Var_\pi(f) \leq \frac{1}{\alpha}\pi\left[\Gamma(f)\right] \hspace{1cm} \forall f\in\R^\cX\;,\\
(ii)& \quad \Var_\pi(f) \leq \frac{1}{4\alpha}\pi\left[\Gamma(f^2, \log f^2)\right] \hspace{1cm} \forall f\in\R^\cX\;.
\end{align*}
\end{lemma}

\begin{proof}
$(i) \Rightarrow (ii)$ just follows from the inequality $\Gamma(f) \leq \frac{\Gamma(f^2, \log f^2)}{4}$.

To prove $(ii) \Rightarrow (i)$, we linearize $(ii)$ taking
$f = 1 + \eps h$ for some $h\in\R^\cX$ with $\pi[h] = 0$ and let
$\eps \rightarrow 0$.
\end{proof}

This equivalence is not true for non-tight versions of the Poincar\'e
inequality, for which we only have $(i) \Rightarrow (ii)$. So we shall
prove non-tight inequalities with the modified Dirichlet form, deduce
a tight inequality with the modified Dirichlet form, and finally
obtain the usual Poincar\'e inequality in the end.

\begin{lemma} \label{lem_non_tight_poincare2} Assume that
  $\Ric(\cX,Q,\pi)\geq -\kappa$ for some $\kappa\geq0$ and the
  diameter of $(\cX,\d_\cW)$ is bounded by $D$. Then for any $\delta > 0$
  and any $f\in\R^\cX$, we have
\begin{equation} \label{non_tight_poincare}
\pi[f^2] \leq \frac{1}{4\delta}\pi\left[\Gamma(f^2, \log f^2)\right] + e^{D^2(\delta + \kappa/2)}\pi\big[\abs{f}\big]^2\;.
\end{equation}
\end{lemma}

\begin{proof}
  Since the distance $\cW$ can be bounded by the $L^2$ Wasserstein
  distance built from $d_\cW$, see \cite[Prop.~3.12]{EM12}, we have
  that $\cW(\rho,\rho')\leq D$ for all $\rho,\rho'\in\cP(\cX)$. After
  multiplying $f$ with a constant we can assume that $\pi[f^2]=1$. The
  $\cH\cW\cI(-\kappa)$ inequality applied to $f^2$ together with
  Young's inequality and the diameter bound yields
\begin{align*}
 \pi[f^2\log f^2]=\cH(f^2) &\leq \cW(f^2,1)\sqrt{\pi\left[\Gamma(f^2,\log f^2)\right]} +\frac{\kappa}{2}\cW(f^2,1)^2\\
          &\leq  \frac{1}{4\delta}\pi\left[\Gamma(f^2, \log f^2)\right] + D^2\left(\delta + \frac{\kappa}{2}\right)\;.
\end{align*}
To obtain the result from this inequality, we can then just follow the
proof of \cite[Lemma 3.5]{GRS11} (The proof uses a different Dirichlet
form, but in this case it makes no difference).
\end{proof}

We will use the following tightening result.

\begin{proposition} \label{prop_tightening_poincare}
Assume that for any function $f\in\R^\cX$ we have
\begin{equation} \label{assump_tightening1}
\Var_{\pi}(f) \leq \alpha_1\pi\left[ \Gamma(f^2, \log f^2)\right] + \beta_1\norm{f}_{\infty}^2
\end{equation}
and
\begin{equation} \label{assump_tightening2}
\Var_{\pi}(f) \leq \alpha_2\pi\left[ \Gamma(f^2, \log f^2)\right] + \beta_2\pi\big[|f|\big]^2
\end{equation}
with constants $\alpha_1,\alpha_2,\beta_1,\beta_2>0$ satisfying
$\frac{\beta_2}{2} + \frac{1}{2}\sqrt{(3\beta_1 + \beta_2-1)(2 +
  \beta_2)} < 1$.
Then the Poincar\'e inequality PI$(\lambda)$ holds with constant
$\lambda = \frac{2 - (\beta_2 + \sqrt{(3\beta_1 + \beta_2-1)(2 +
    \beta_2)})}{8(3\alpha_1 + \alpha_2)}$.
\end{proposition}

Combining the weak Poincar\'e inequality, Lemma
\ref{lem_non_tight_poincare2} and this tightening result, after
optimizing the constants, we immediately obtain

\begin{theorem}\label{thm:Poincare}
  Assume that $\Ric(\cX,Q,\pi)\geq -\kappa $ for some $\kappa \geq 0$
  and that the diameter of $(\cX,d_\cW)$ is bounded by $D$. Assume
  that
  $e^{D^2\kappa/2} + \sqrt{(3\kappa D^2/2 + e^{D^2\kappa/2} -1)(2 +
    e^{D^2\kappa/2})} < 2$.
  Then the Poincar\'e inequality PI$(\lambda)$ holds with a constant
  $\lambda$ that only depends on $\kappa$ and $D$. In particular, if
  $\Ric(\cX,Q,\pi)\geq0$, then PI$(c D^{-2})$ holds for a universal
  constant $c$.
\end{theorem}

The best possible value of the constant $c$ we obtain with this proof
is hard to determine, but we can show that it satisfies
$c \geq \frac{9 - \sqrt{62}}{80(45 + \ln (11/10))}$. In Section
\ref{subsection_alt_sg} we present an alternative argument that yield
PI$(cD^{-2})$ with an explicit and probably better constant $c$.

\begin{proof} [Proof of Proposition \ref{prop_tightening_poincare}]
  The proof essentially follows the argument of
  \cite[Prop.~7.5.6]{BGL15}, except that we use the Dirichlet form
  $\pi\left[\Gamma(f^2, \log f^2)\right]$, so some adaptation is
  required.

  Consider $f$ satisfying $\operatorname{med}(f) = 0$ and
  $\pi[f^2] = 1$ and fix $R > 0$. Without loss of generality, we
  may assume that $\pi[ \{f = 0\}] =0$, otherwise
  $\pi\left[\Gamma(f^2, \log f^2)\right] = \infty$ and there is nothing to
  prove. Let $f_R(x) = f(x)$ if $|f(x)| < R$, and $R$ (resp. $-R$) if
  $f(x) \geq R$ (resp. $f(x) \leq -R$), and $d_R := \pm R + f - f_R$
  (depending on the sign of $f(x)$). We have that both
  $\Gamma(f_R^2, \log f_R^2)$ and $\Gamma(d_R^2, \log d_R^2)$ are smaller
  than $\Gamma(f^2, \log f^2)$. Notice that $f_R$ also has median $0$.

  Now we have
  \begin{align*}
    1 &= \pi[f^2] = \pi[(f_R + (f - f_R))^2] \\
      &= \pi[f_R^2] + 2\pi[ f_R (f - f_R)] + \pi[ (f - f_R)^2]\\
      &= \pi[f_R^2] + \pi[ d_R^2] - R^2 
  \end{align*}
  We also have
  \begin{align*}
    \left|\pi[f - f_R] \right| &\leq \pi\big[ |f - f_R|\big] = \pi\big[ |f|\mathbf{1}_{|f| \geq R}\big] - R\pi[\{|f| \geq R\}] \\
                                     &\leq \sqrt{\pi[f^2]} \sqrt{\pi[\{|f| \geq R\}]} - R\pi[\{|f| \geq R\}] \leq \frac{1}{4R}.
  \end{align*}
  Moreover, since $f$ has median $0$, we have
  \begin{equation}
    \left|\pi[ d_R]\right| = \left|\pi[f - f_R] \right| \leq \frac{1}{4R}.
  \end{equation}
  Since $f_R$ has median $0$, we have $\pi[f_R^2] \leq 3\Var_{\pi}(f)$
  (see for instance \cite[Lem.~2.1]{Mil09}). Applying
  \eqref{assump_tightening1} to $f_R$, we then get
  \begin{equation} \label{eq-bnd-fR} \pi[f_R^2] \leq 3\Var_{\pi}(f)
    \leq 3\alpha_1\pi\left[\Gamma(f^2, \log f^2)\right] + 3\beta_1R^2.
  \end{equation}
  We now seek to bound $\pi[d_R^2] - R^2$. Applying
  \eqref{assump_tightening2}, we have
  \begin{align*}
\pi[ d_R^2] &\leq \pi[d_R]^2 + \alpha_2\pi\left[\Gamma(f^2, \log f^2)\right] + \beta_2\pi\big[|d_R|\big]^2\\
&\leq \alpha_2\pi\left[\Gamma(f^2, \log f^2)\right] + \beta_2\pi\big[|d_R|\big]^2 + \frac{1}{16R^2}\;.
  \end{align*}
Since
\begin{align*}
\pi\big[|d_R|\big]^2 &= R^2 + 2R\pi\big[ |f-f_R|\big] + \pi\big[ |f-f_R|\big]^2 \\
&\leq R^2 + \frac{1}{2} + \frac{1}{16R^2},
\end{align*}
we get
\begin{equation}
\pi[d_R^2] - R^2 \leq  \alpha_2\pi\left[\Gamma(f^2, \log f^2)\right]  + \frac{\beta_2}{2} + (\beta_2-1)R^2 + \frac{1 + \beta_2}{16R^2}\;.
\end{equation}
Combining this estimate with \eqref{eq-bnd-fR} we obtain
\begin{align*}
1 = \pi[f^2] \leq (3\alpha_1 + \alpha_2)\pi\left[\Gamma(f^2, \log f^2)\right] + \frac{\beta_2}{2} + (3\beta_1 + \beta_2-1)R^2 + \frac{2 + \beta_2}{16R^2}\;.
\end{align*}
Optimizing in $R$ then yields 
\begin{align*}
1 \leq (3\alpha_1 + \alpha_2)\pi\left[\Gamma(f^2, \log f^2)\right] + \frac{\beta_2}{2} + \frac{1}{2}\sqrt{(3\beta_1 + \beta_2-1)(2 + \beta_2)}\;.
\end{align*}
Since $\Var_{\pi}(f) \leq \pi[f^2] = 1$, this amounts to the
Poincar\'e inequality by Lemma \ref{lem:poinc-comp} as soon as
\begin{align*}
\frac{\beta_2}{2} + \frac{1}{2}\sqrt{(3\beta_1 + \beta_2-1)(2 +
  \beta_2)} < 1\;.
\end{align*}
\end{proof}

We can use the same arguments to treat the case where the diameter is
not necessarily bounded, but the distance $d_{\mathcal{W}}$ has a
square-exponential moment:

\begin{theorem}
  Assume that $\Ric(\cX,Q,\pi)\geq 0$, and that there exists a
  constant $\alpha > 0$ such that for some $x_0\in\cX$
  \begin{align*}
   D_{\alpha}:=\pi\left[ e^{\alpha d_{\mathcal{W}}(\cdot, x_0)^2}\right] < \infty\;.
  \end{align*}
  Then the Poincar\'e inequality PI$(\lambda)$
  holds with some constant $\lambda$ which depends on $\alpha$ and on
  the value $D_{\alpha}$ of the above expectation.
\end{theorem}

Note that the finiteness of the integral does not depend on the choice
of $x_0$, but the value does, which affects the value of the constant $\lambda$ we obtain.

\begin{proof}
  Since the proof follows the same lines as for the bounded diameter
  case, we shall only sketch it and point out the extra arguments
  required. Since we have a square-exponential moment, a Gaussian
  concentration property (and hence an exponential concentration
  property) holds, and Proposition \ref{prop_weak_poincare} still
  applies. So all we need to do is to show that the conclusion of
  Lemma \ref{lem_non_tight_poincare2} still holds. Let $f$ be a
  probability density. Note that by convexity of $\cW^2$ we have the bound
\begin{align*}
\mathcal{W}&(f, 1) ^2 \leq \underset{x,y}{\sum} \hspace{1mm} d_{\mathcal{W}}(x,y)^2f(x)\pi(x)\pi(y) \\
&\leq 2\pi\left[d_{\mathcal{W}}(\cdot,x_0)^2f(\cdot)\right] + 2\pi\left[d_{\mathcal{W}}(x_0,\cdot)^2\right]\;.
\end{align*}
The second term is a constant that does not depend on $f$, and can be
bounded using only the square-exponential moment. For the first term,
we can use the bound
\begin{align*}
\pi\left[d_{\mathcal{W}}(\cdot,x_0)^2f(\cdot)\right] \leq \frac{1}{\alpha}\log\pi\left[ e^{\alpha d_{\mathcal{W}}(\cdot, x_0)^2}\right] + \frac{1}{\alpha}\pi[f\log f]\;.
\end{align*}
Combining this inequality with $\cH\cW\cI(0)$ and Young's inequality
yields for any $\delta>0$:
\begin{align*}
  \pi[f\log f] 
  &\leq \frac{1}{4\delta} \pi\left[\Gamma(f,\log f)\right] + 2\delta \cW(f,1)^2\\
 &\leq \frac{1}{4\delta} \pi\left[\Gamma(f,\log f)\right] + \frac{2\delta}{\alpha}\pi[f\log f] + C\delta
\end{align*}
with a constant $C$ that depends on $\alpha$ and the
square-exponential moment $D_{\alpha}$, but not on $f$. Since we can
make $\delta$ arbitrarily small, the second term on the right-hand
side can be absorbed into the left-hand side. Then the proof continues
by applying this inequality to $f^2$ with $\pi[f^2]=1$ and arguing in
the same way as for the bounded diameter case.
\end{proof}

\subsection{An alternative argument}

\label{subsection_alt_sg}

We present an alternative derivation of the Poincar\'e inequality from
diameter bounds in non-negative curvature following the approach of
\cite{CLY}.

\begin{proposition}\label{prop:alaCLY}
  Assume that $\Ric(\cX,Q,\pi)\geq 0$ and that the diameter of
  $(\cX,d_\cW)$ is bounded by $D$. Then the Poincar\'e inequality
  PI$(\lambda)$ holds with
  \begin{align*}
    \lambda =\frac1{eD^2}\;.
  \end{align*}
\end{proposition}

\begin{proof}
  Recall that the optimal constant in the Poincar\'e inequality is
  (minus) the first non-zero eigenvalue of the generator $L$. Let $f$
  be an eigenfunction of the $L$ with eigenvalue $-\lambda_1$. By
  scaling, we can assume without loss of generality that
  $\norm{f}_\infty=1$.  Since $f$ necessarily satisfies $\pi[f]=0$, we
  have $\min f < 0 < \max f$.

  Note that $P_tf=e^{-\lambda_1 t}f$. Thus, the reverse Poincar\'e inequality
  \eqref{eq:reverse-Poincare} implies that for any $\rho\in\cP(\cX)$
  we have
  \begin{align*}
    \abs{\nabla f}^2_\rho \leq \frac{e^{2\lambda t}}{2t}\norm{f}_\infty^2\;.
  \end{align*} 
  Optimizing in $t$ and using $\norm{f}_\infty=1$ we find that
  \begin{align*}
\abs{\nabla f}^2_\rho \leq e\lambda_1\;.
  \end{align*}
  Now, let $x_0,x_1$ be such that $f(x_0) = \min f$ and
  $f(x_1) = \max f$. Let $\eps>0$ and let
  $(\rho_s,\psi_s)_{s\in[0,1]}$ be a curve satisfying \eqref{eq:cont}
  such that
  $\int_0^1\abs{\nabla\psi_s}^2_{\rho_s}\dd s\leq \cW(\delta_{x_0},\delta_{x_1})^2+\eps$. Then we estimate
  \begin{align*}
    1&\leq [f(x_1)-f(x_0)]^2 =
   \left|\sum f\rho_1\pi - \sum f\rho_0\pi\right|^2 = \left|\int_0^1\langle\nabla f,\nabla \psi_s\rangle_{\rho_s}\dd s\right|^2\\
    &\leq  (D^2+\eps) \int_0^1\abs{\nabla f}^2_{\rho_s}\dd s \leq (D^2+\eps) \lambda_1 e\;
  \end{align*}
	and the result immediately follows. 
\end{proof}

This argument could be adapted to treat the case of negative entropic
Ricci curvature as well, but we will not pursue this. We note that the
proof just given is quite simpler than the previous one. However, it
is not clear that we can use the same argument to cover the case where
we only assume the invariant measure to have a square-exponential
moment, or how to use it to prove a modified logarithmic Sobolev
inequality. We will see in the next section that this is possible with
the first method.

\section{Modified logarithmic Sobolev inequalities}
\label{sec:mlsi}

In this section we will prove the third main result Theorem
\ref{main_thm_mlsi} establishing a modified logarithmic Sobolev
inequality for Markov chains with non-negative entropic Ricci
curvature under a diameter bound. We will then apply this to derive
bounds on the total variation mixing time of the Markov
chain. Finally, we formulate conjectures about possible improvements
of the results replacing the bound on the diameter with a control on
concentration properties.

\subsection{Modified LogSobolev inequality from diameter bounds}

We will show the following

\begin{theorem} \label{thm:MLSI} Assume that $\Ric(\cX,Q,\pi)\geq 0$
  and that the diameter of $(X, d_{\mathcal{W}})$ is bounded by
  $D$. Then the modified logarithmic Sobolev inequality
  MLSI$(\lambda)$ holds with constant $\lambda=\frac{c}{D^2}$ for some
  universal constant $c$.
\end{theorem}

For convenience, we shall reformulate the modified logarithmic Sobolev
inequality so that it applies to arbitrary non-negative functions
instead of probability densities. To this end, given a measure $\nu$
and a function $g\in\R_+^\cX$ we define
\begin{align*}
  \Ent_{\nu}(g) = \nu[g\log g] - \nu[g]\log\nu[g]\;.
\end{align*}
It is then immediate to check that the inequality MLSI$(\lambda)$
defined in \eqref{eq:MLSI} is equivalent to
\begin{align*}
\Ent_{\pi}(f^2) \leq \frac{1}{2\lambda}\pi\left[\Gamma(f^2, \log f^2)\right]\quad \forall f\in\R^\cX\;.
\end{align*}
The proof of Theorem \ref{thm:MLSI} will again consist in first
obtaining a weak version of the MLSI via the $\cH\cW\cI$ inequality,
and then tightening it. In the continuous setting, the corresponding result (and
actually a much stronger one, as we shall discuss in the next section)
was proven employing such a strategy in \cite{GRS11}. That work
strongly relies on a self-tightening property of the logarithmic
Sobolev inequality, which states that if a non-tight LSI of the form
$$\Ent_{\pi}(f^2) \leq cI(f) + \alpha$$
holds and if $\alpha$ is small enough, then a tight LSI holds. It is
not clear whether such a strong self-tightening property holds for the
discrete modified logarithmic Sobolev inequality. To bypass this
issue, we have to rely on more involved arguments, inspired by a
work of Barthe and Kolesnikov \cite{BK}.

We shall need the following two lemmas.

\begin{lemma} \label{lem_ent_bnd_hwi} Assume that
  $\Ric(\cX,Q,\pi)\geq0$ and that the diameter of $(\cX,d_\cW)$ is
  bounded by $D$. Then we have for any $\delta > 0$ and $f\in\R^\cX$:
  \begin{align}\label{eq:nontight-mlsi}
\Ent_{\pi}(f^2) \leq \delta D^2 \pi\left[\Gamma(f^2, \log f^2)\right] +
  \frac{1}{4\delta}\pi\left[f^2\one_{\{f^2 > \pi[f^2]\}}\right]\;.
  \end{align}
\end{lemma}

\begin{proof}
  Since \eqref{eq:nontight-mlsi} is homogeneous in $f^2$ we can assume
  without restriction that $\pi[f^2]=1$. From the $\cH\cW\cI(0)$
  inequality and Young's inequality we infer
  \begin{align}\label{eq:hwi-applied}
\pi[f^2\log f^2] \leq \delta D^2 \pi\left[ \Gamma(f^2, \log f^2)\right] + \frac{1}{4\delta D^2}\cW(f^2,1)^2\;.
  \end{align}
  We know that $\mathcal{W} \leq W_{2, d_{\mathcal{W}}}$. It is not
  hard to construct a transport from $\mu = f^2\pi$ to $\pi$ that
  moves mass away only from points $x$ where $\mu(x) > \pi(x)$,
  i.e.~$f^2(x)>1$. Hence we have that
  $W_{2, d_{\mathcal{W}}}\left(\mu, \pi\right)^2 \leq
  D_{\mathcal{W}}^2\pi\left[f^2\one_{\{f^2>1\}}\right]$,
  and the result immediately follows.
\end{proof}

\begin{lemma}[{\cite[Lem.~2.5]{BK}}] \label{lem_bk}
For any $A > 1$ there exists $\gamma > 0$ such that for any $f\in\R^\cX$ with $\pi[f^2] = 1$, we have
\begin{align}\label{eq:bk1}
\pi\left[f^2\one_{\{f^2 \geq A^2\}}\right] &\leq \left(\frac{A}{A-1}\right)^2\Var_{\pi}(f)\;,\\\label{eq:bk2}
\pi[f^2\log f^2] &\leq \gamma \Var_{\pi}(f) + \pi\left[f^2\log f^2\one_{\{f^2 \geq A^2\}}\right]\;.
\end{align}
\end{lemma}

We can now give the proof of our third main result Theorem
\ref{main_thm_mlsi}.

\begin{proof}[Proof of Theorem \ref{thm:MLSI}]
  Fix $A > 1$ and $f\in\R^\cX$ with $\pi[f^2]= 1$. Set
  $f_A(x) :=\max\big(f(x),A\big)$, and define the probability measure
  $\mu_A=f_A^2/Z_A\pi$, where $Z_A :=\pi[f_A^2]$. Note that
  $A^2\leq Z_A\leq 1+A^2$. From \eqref{eq:bk2} in Lemma \ref{lem_bk} we
  have
  \begin{align}\label{eq:est1}
\Ent_{\pi}(f^2) \leq \gamma \Var_{\pi}(f) + \pi\left[f^2\log f^2\one_{\{f^2 \geq A^2\}}\right]\;.
  \end{align}
  The first term can be estimated via the Poincar\'e inequality as
\begin{align}\label{eq:est2}
  \gamma \Var_{\pi}(f) \leq \gamma\frac{c}{4D^2}\pi\left[\Gamma(f^2,\log f^2)\right]
\end{align}
using Theorem \ref{thm:Poincare} and Lemma \ref{lem:poinc-comp}. For
the second term, we have
\begin{align}\nonumber
\pi\left[f^2\log f^2\one_{\{f^2 \geq A^2\}}\right] &= \pi\left[ f_A^2\log f_A^2\right] - A^2\log A^2 \pi[\{f < A\}]\\\label{eq:est3}
&= Z_A\Ent_{\pi}(\mu_A) + Z_A \log Z_A- A^2\log A^2 \pi[\{f < A\}]\;.
\end{align}
The entropy term in \eqref{eq:est3} can be handled using Lemma
\ref{lem_ent_bnd_hwi} and the fact that we have
$\Gamma(f^2_A, \log f^2_A) \leq \Gamma(f^2, \log f^2)$.
\begin{align*}
\Ent_{\pi}(\mu_A) &\leq \frac{\delta D^2}{Z_A}\pi\left[\Gamma(f^2, \log f^2)\right] + \frac{1}{Z_A}\frac{1}{4\delta}\pi\left[f^2\one_{\{f_A^2 \geq Z_A\}}\right]\\
&\leq  \frac{\delta D^2}{A^2}\pi\left[\Gamma(f^2, \log f^2)\right] + \frac{1}{4\delta A^2} \pi\left[f^2\one_{\{f^2 \geq A^2\}}\right]\;.
\end{align*}
Using further \eqref{eq:bk1} and again the Poincar\'e inequality via
Theorem \ref{thm:Poincare} and Lemma \ref{lem:poinc-comp} we arrive at
\begin{align}\nonumber
\Ent_{\pi}(\mu_A)&\leq \frac{\delta D^2}{A^2}\pi\left[\Gamma(f^2, \log f^2)\right] + \frac{1}{4\delta(A-1)^2}\Var_{\pi}(f)\\\label{eq:est4}
                & \leq \left(\frac{\delta D^2}{A^2} + \frac{1}{4\delta(A-1)^2}\frac{c}{4D^2}\right)\pi\left[\Gamma(f^2, \log f^2)\right]\;.
\end{align}
So all that is left is to bound is the term
$Z_A \log Z_A- A^2\log A^2 \pi[\{f < A\}]$. We have
\begin{align*}
&Z_A \log Z_A- A^2\log A^2 \pi[\{f < A\}]\\
& = \Big(A^2\pi[\{f<A\}] + \pi\left[ f^2\one_{\{f^2 \geq A^2\}}\right]\Big) \log \Big(A^2\pi[\{f<A\}] + \pi\left[f^2\one_{\{f^2 \geq A^2\}}\right]\Big) \\
& \hspace{5mm} - A^2\log A^2 \pi[\{f < A\}]\\
&= A^2\pi[\{f<A\}]\log\left(\pi[\{f<A\}] + \frac{\pi\left[f^2\one_{\{f^2 \geq A^2\}}\right]}{A^2}\right) \\
&\hspace{5mm} + \pi\left[f^2\one_{\{f^2 \geq A^2\}}\right] \times \log \Big(A^2\pi[\{f<A\}] + \pi\left[f^2\one_{\{f^2 \geq A^2\}}\right]\Big) \\
&\leq A^2\log\Big(1 + \frac{\pi\left[f^2\one_{\{f^2 \geq A^2\}}\right]}{A^2}\Big) + \log(1 + A^2)\pi\left[f^2\one_{\{f^2 \geq A^2\}}\right] \\
& \leq \big(1 + \log(1 + A^2)\big)\pi\left[f^2\one_{\{f^2 \geq A^2\}}\right]\;.
\end{align*}
We can then once more use \eqref{eq:bk1} from Lemma \ref{lem_bk} to
bound this by the variance, and then the Poincar\'e inequality to arrive at
\begin{align}\nonumber
  &Z_A \log Z_A- A^2\log A^2 \pi[\{f < A\}]\\\label{eq:est5}
  &\leq \big(1 + \log(1 + A^2)\big) \left(\frac{A}{A-1}\right)^2\frac{c}{4D^2}\pi\left[\Gamma(f^2,\log f^2)\right]\;.
\end{align}
Combining \eqref{eq:est1} with \eqref{eq:est2}, \eqref{eq:est3},
\eqref{eq:est4} and \eqref{eq:est5} finishes the proof.
\end{proof}

\begin{remark}
  By the same method one can obtain a modified logarithmic Sobolev
  inequality under the assumption that $\Ric(\cX,Q,\pi)\geq-\kappa$
  and for $\kappa>0$ that the diameter is bounded by $D$, provided
  $\kappa$ is sufficiently small compared to $D$. The only
  modification is an extra term $\kappa/2\cW(f^2,1)^2$ appearing in
  the application of the $\cH\cW\cI(\kappa)$ inequality in
  \eqref{eq:hwi-applied}. A similar remark applies to the next result.
\end{remark}

As for the Poincar\'e inequality, we can replace the diameter bound by
a finite square-exponential moment.

\begin{theorem} \label{thm_mlsi_conc} Assume that
  $\Ric(\cX,Q,\pi)\geq 0$ and that there exists a constant
  $\alpha > 0$ such that
  \begin{align*}
    D_{\alpha}=\pi\left[ e^{\alpha d_{\mathcal{W}}(\cdot, x_0)^2}\right] < \infty
  \end{align*}
  for some $x_0\in\cX$. Then the modified logarithmic Sobolev
  inequality MLSI$(\lambda)$ holds with some constant $\lambda$ which
  depends on $\alpha$ and on the value $D_\alpha$ of the integral.
\end{theorem}

The proof proceeds in exactly the same way the proof of Theorem
\ref{thm:MLSI} except that we need the following replacement for Lemma
\ref{lem_ent_bnd_hwi}.

\begin{lemma}\label{lem:ent_bnd_hwi2}
  Under the assumptions of Theorem \ref{thm_mlsi_conc} there exists a
  constant $C=C(\alpha, D_\alpha)$ depending only on $\alpha$ and
  $D_\alpha$ such that for any $\delta > 0$ and $f\in\R^\cX$ we have:
  \begin{align}\label{eq:nontight-mlsi}
\left(1-\frac{\delta}{\alpha}\right)\Ent_{\pi}(f^2) \leq \frac1{4\delta} \pi\left[\Gamma(f^2, \log f^2)\right] +
  \delta C \pi\left[f^2\one_{\{f   > \pi[f^2]\}}\right]\;.
  \end{align}
\end{lemma}

\begin{proof}
  Without restriction we can assume that $\pi[f^2]=1$.  Arguing as in
  the proof of Lemma \ref{lem_ent_bnd_hwi} we obtain the crude bound
  \begin{align*}
    \cW(f^2,1)^2 \leq \sum_{x,y}\d_\cW(x,y)^2f^2(x)\one_{\{f>1\}}(x)\pi(x)\pi(y)\;.
  \end{align*}
  From the $\cH\cW\cI(0)$ inequality we thus infer that
\begin{align*}
\Ent_{\pi}(f^2) &\leq \frac{1}{4\delta}\pi\left[\Gamma(f^2,\log f^2)\right] + \delta\underset{x,y}{\sum} \hspace{1mm} d_\cW(x,y)^2f^2(x)\one_{\{f   > 1\}}(x)\pi(x)\pi(y)\;.
  \end{align*}
  From the triangle inequality we have
\begin{align*}
\sum d_\cW(x,y)^2f^2(x)\one_{\{f   > 1\}}\pi(x)\pi(y) \leq 2\pi\left[ d_\cW(x_0,\cdot)^2f^2\one_{\{f > 1\}}\right]+2\pi\left[f^2\one_{\{f  >1\}}\right]\pi\left[d(\cdot, x_0)^2\right]\;.
\end{align*}
The bound on the exponential moment immediately leads to a bound of
the form $C\pi\left[f^2\one_{\{f > 1\}}\right]$ for the second term on
the right-hand side, with $C$ only depending on $\alpha$ and
$D_\alpha$. We thus consider the first term. From the Young-type
inequality $ab \leq a \log a + e^b$ for $a \geq 0$ and $b \in \R$ we
deduce, setting $Z = \pi\left[f^2\one_{\{f > 1\}}\right]$, that
\begin{align*}
\pi\left[ d_\cW(x_0,\cdot)^2f^2\one_{\{f > 1\}}\right] &= \frac{Z}{\alpha} \pi\left[\alpha d_\cW(x_0,\cdot)^2\frac{f^2}{Z}\one_{\{f > 1\}}\right]\\
&\leq \frac{Z}{\alpha} \Ent_{\pi \one_{\{f > 1\}}}\left(\frac{f^2\one_{\{f > 1\}}}{Z}\right) + \frac{Z}{\alpha}\pi\left[ e^{\alpha d_\cW(\cdot, x_0)^2}\one_{\{f > 1\}}\right] \\
&\leq  \frac{1}{\alpha}\Ent_{\pi \one_{\{f > 1\}}}(f^2\one_{\{f > 1\}}) + \frac{D_\alpha}{\alpha}Z\;.
\end{align*}
Hence the proof is finished once we note that
$\Ent_{\pi \one_{\{f > 1\}}}(f^2\one_{\{f > 1\}}) \leq
\Ent_{\pi}(f^2)$. This is a consequence of the duality formula
\begin{align*}
\Ent_{\nu}(g) = \underset{h}{\sup}\Big[\nu[hg] - \log \nu\left[e^{h}\right] + \log \nu[\cX]\Big]
\end{align*}
for any non-negative function $g$ with $\nu[g]=1$.
\end{proof}

\subsection{Total variation mixing time for Markov chains with non-negative curvature}

Both the Poincar\'e inequality and the logarithmic Sobolev inequality
yield bounds on the rate of convergence to equilibrium for the Markov
chain, respectively in the $L^2(\pi)$ norm and in relative entropy,
see Section \ref{sec:fi} and in particular
\eqref{eq:exp-conv}. Another relevant way of measuring closeness to
equilibrium, often used in practice, is the total variation norm. In
particular, there is a lot of interest in obtaining bounds on the
total variation mixing time, defined as follows.

\begin{definition}
  The total variation mixing time is defined for $\eps>0$ as
  \begin{align*}
\tau_{mix}(\eps) := \sup \Big\{t > 0; \norm{P_t^*\delta_x - \pi}_{{\rm TV}} < \eps\; \forall x\in\cX\Big\}\;.
  \end{align*}
\end{definition}

Here $P_t^*$ denotes the dual Markov semigroup acting on probability
measures. We refer to the book \cite{LPW} for an introduction and
overview of the many works on mixing times.

Since the Pinsker inequality states that
$2\norm{\nu - \pi}_{\rm TV}^2 \leq \Ent_{\pi}(\nu)$, the modified
logarithmic Sobolev inequality is a useful tool to obtain upper bounds
on the mixing time. However, since the estimate must hold uniformly in
the initial data, it is not enough. In the continuous setting, since
the relative entropy functional is unbounded, an extra argument is
always needed. In the finite setting, since we always have
$\Ent_{\pi}(\delta_x) = -\log \pi(x)$, the inequality MLSI$(\rho)$ implies the estimate
\begin{align*}
\tau_{mix}(\eps) \leq \frac{1}{2\rho}\left[-\log (2\eps^2) + \log \log \pi_*^{-1}\right]
\end{align*}
where $\pi_* = \inf \{\pi(x): x\in\cX\}$.

One of the flaws of this bound is that $\pi^*$ is quite small when the
space has many points. In particular, it does not behave well when
studying continuous limits. In the context of Markov chains with
non-negative curvature, we can give a general estimate on the mixing
time that does not involve $\pi_*$.

\begin{theorem}
  Assume that $\Ric(\cX,Q,\pi)\geq 0$ and that the diameter of
  $(\cX,d_\cW)$ is bounded by $D$. If MLSI$(\rho)$
  holds then we have
  \begin{align*}
\tau_{mix}(\eps) \leq \frac{D^2}{4} + \frac{\log \epsilon}{\rho}\;.
  \end{align*}
  In particular, we obtain that for a universal constant $c$
  \begin{align*}
    \tau_{mix}(\eps) \leq D^2(1/4 + c\log \eps)\;.
  \end{align*}
\end{theorem}

\begin{proof}
  The second bound immediately follows from the first using that
  Theorem \ref{thm:MLSI} yields the validity of MLSI$(cD^{-2})$ for a
  suitable constant $c$. The show the first bound we first note the estimate
  \begin{align*}
 \cH(P_t f) \leq \frac{\mathcal{W}(f, 1)^2}{4t}
  \end{align*}
  which is an immediate consequence of the Evolution Variational
  Inequality established in \cite[Thm.~4.5]{EM12}. Hence
  $\cH(P_tf)\leq 2$ for all $t\geq D^2/4$ and all $f\in\PX$. The
  result then follows using the exponential convergence
  $\cH(P_tf)\leq e^{-2\rho t}\cH(f)$ implied by MLSI$(\rho)$ and
  Pinsker's inequality.
\end{proof}

\subsection{A conjecture}

\label{sect_milman}

If we apply the abstract results to a simple random walk on the
discrete torus $(\Z/L\Z)^d$, we get a spectral gap and a modified LSI
with constant $O(d^2L^2)$. However, the optimal constant behaves like
$dL^2$, so our estimate is off by a dimensional factor. This was to be
expected: if we consider a product space, both the Poincar\'e
inequality and the modified LSI tensorize (up to a scaling of the
time), while the squared diameter grows linearly with the
dimension. This shows that diameter estimates should not allow one to
capture the sharp behavior of functional inequalities for dynamics in
high dimension.

To have any hope of obtaining good estimates in high dimension, we
should therefore rely on a different kind of assumption. In a series
of contributions \cite{Mil09, Mil09b, Mil10}, Milman showed that
for Riemannian manifolds, we can effectively use assumptions on the
concentration profile to derive functional inequalities for manifolds
of non-negative Ricci curvature. This improves on the diameter assumption,
since concentration estimates may be dimension-free (although not
always). Moreover, it is a strictly weaker assumption, since when the
diameter is bounded we automatically have Gaussian and exponential
concentration, with constants controlled by the diameter.

More precisely, what Milman showed is the following: 
\begin{itemize}
\item If curvature is bounded from below by $-\kappa$ for some $\kappa
  > 0$, then a strong enough Gaussian concentration implies a Gaussian
  isoperimetric inequality, and hence both a logarithmic Sobolev
  inequality and a Poincar\'e inequality. The constant only depends on
  $\kappa$ and on the constant appearing in the Gaussian concentration
  property.

\item If curvature is non-negative, exponential concentration implies a
  Cheeger isoperimetric inequality, and hence a Poincar\'e
  inequality. The constant only depends on the constant appearing in
  the exponential concentration property.
\end{itemize}

Since Gaussian concentration is equivalent to finiteness of a
square-exponential moment, qualitatively the first result at first
glance may not appear so different from Wang's theorem. The important
difference (in addition to the isoperimetric inequality) is that the
constant does not depend anymore on the value of the
square-exponential moment. This makes a significant difference in high
dimensional situations, where the square exponential moment depends on
the dimension, but the Gaussian concentration constant often does not.

Milman's work relies on tools of Riemannian geometry (concavity of
isoperimetric profiles and the Heinz-Karcher theorem), so it does not
seem like his arguments can be adapted to the discrete case. An
alternative proof by Ledoux \cite{Le11} also relied on concavity of
isoperimetric profiles.

As we have seen in the previous sections, the alternative approach of
Gozlan, Roberto and Samson \cite{GRS11}, based on functional
inequalities, is more easily adapted to the discrete setting. While
unlike Milman, they do not recover the Gaussian isoperimetric
inequality, they nonetheless show that when curvature is bounded from
below, a strong enough Gaussian concentration implies a logarithmic
Sobolev inequality. However, we have not been able to adapt a key step
in their approach, which is that Gaussian concentration implies a weak
transport-entropy inequality. In the discrete setting, the analogous
inequality we would need would be
$$\mathcal{W}(\mu, \pi)^2 \leq c_1\Ent_{\pi}(\mu) + c_2.$$
To establish it, we would need to better understand the relationship
between bounds on $\mathcal{W}$ and concentration. An important
difference between the discrete and the continuous situation is that
lack of a dual Kantorovich formulation for the distance $\mathcal{W}$.

Nonetheless, we state as conjectures the discrete analogues of the
results of \cite{Mil09, Mil09b, Mil10, Le11, GRS11}:

\begin{conjecture}
  Assume that $\Ric(\cX,Q,\pi)\geq0$ and that the invariant measure
  $\pi$ satisfies a concentration property w.r.t.~the distance
  $d_{\mathcal{W}}$ with profile $\alpha(r) = Me^{-\rho r}$. Then
  there exists a constant $C(M)$ such that PI$\big(C(M)\rho^{-2}\big)$ holds.
\end{conjecture}

\begin{conjecture} \label{conjecture2} Assume that
  $\Ric(\cX,Q,\pi)\geq-\kappa$ for some $\kappa > 0$, and that a
  concentration property with respect to the distance
  $d_{\mathcal{W}}$ holds with profile $\alpha(r) = Me^{-\rho r^2}$.
  Then there exists a constant $\tau(M)$ and $\lambda(\kappa,M,\rho)$
  such that if $\frac{\kappa}{\rho} < \tau(M)$ then
  MLSI$\big(\lambda(\kappa,M,\rho)\big)$ holds. If moreover
  $\Ric(\cX,Q,\pi)\geq0$ then MLSI$(cM\rho)$ holds for some universal
  constant $c$.
\end{conjecture}

In the Riemannian setting, these results hold with no dependence on
$M$, but for non-smooth geodesic spaces the proof of \cite{GRS11} has
an extra dependence on $M$ of the form we use in the statements of
these conjectures.

As in the continuous setting, Theorem \ref{thm_mlsi_conc} already tells us that
under these assumptions a mLSI holds. The open problem in Conjecture
\ref{conjecture2} is the value of the constant.

As we shall see in the next section, if these conjectures are indeed
true, then we could use curvature to better understand the behavior of
some interacting particle systems with degenerate rates.

\section{Application to the zero-range process with constant rates}
\label{sec:zero-range}

In this section, we shall discuss functional inequalities for a system
of $K$ interacting particles on the complete graph with $L$ sites,
namely the zero range process.

The state space is $\cX_{K,L}=\{\eta\in\N^L:\sum_{i=1}^L\eta_i=K\}$. The
dynamics we are interested in is defined as follows. With rate $1$, we
select a site $i$ uniformly at random. If $\eta_i = 0$ (no particles on
site $i$), we do nothing. Else we choose a second site $j$ uniformly
at random, and move a single particle from $i$ to $j$. We shall denote
by $\eta^{i,j}$ the new configuration obtained after such a move. More
precisely, the transition rates of the corresponding continuous time
Markov chain for $\eta\neq \eta'$ are thus given by
\begin{align*}
  Q_{K,L}(\eta,\eta')=
  \begin{cases}
    \frac{1}{L}    &  \eta'=\eta^{i,j} \text{ for some } i,j\;,\\
    0\;, & \text{else}\;.
  \end{cases}
\end{align*}
The invariant measure is the uniform measure on $\cX_{K,L}$ denoted by
$\pi_{K,L}$.

This model constitutes a degenerate version of the classical zero
range process, where particles on site $i$ jump at rate $f(\eta_i)$
for some rate function $f$. For example, independent particles
correspond to the case $f(n) = \lambda n$ for a constant
$\lambda$. Our situation corresponds to the case where the jump rate
$f$ is constant.

In \cite{FM15}, entropic Ricci curvature lower bounds for the zero
range process were established, in the situation where the jump rate
is strictly increasing: If the rate satisfies
$0 < c \leq f(n+1) - f(n) \leq c + \delta$ for all $n$ and some
constants $c,\delta$ and $\delta$ is small enough compared to $c$,
then curvature is bounded from below by a strictly positive
constant. It is easy to check that the proof can be straightforwardly
adapted to show that the zero range process with constant rates has
non-negative curvature,
i.e.~$\Ric(\cX_{K,L},Q_{K,L},\pi_{K,L})\geq0$. We can thus use the
abstract results of the previous section together with the following
diameter estimate to obtain the mLSI for the degenerate zero range
process.

\begin{lemma} \label{lem_diam_zrp} There exists a constant $c > 0$
  such that for any $L, K$ and the diameter of $(\cX_{K,L},d_\cW)$ is
  bounded by $c K \sqrt{L\log L}$.
\end{lemma}

\begin{theorem}
  For the zero-range process with constant rate $1$ on the complete
  graph with $L$ sites, $K$ particles the modified logarithmic Sobolev
  inequality MLSI$\big(\frac{c}{K^2L\log L}\big)$ for a universal constant $c$.
\end{theorem}

We do not believe this constant to be optimal. Morris \cite{Mor06}
showed that the spectral gap is of order $L/K^2$, so for a fixed
density of particles $K/L$ our estimate is off by a factor
$K^2\log L$. For the mLSI, no better result seems to be known, but we
believe that it should behave like $1/L$ at fixed density, by analogy
with the situation for gamma distributions studied in \cite{BW}. As
mentioned in Section \ref{sect_milman}, one source of error is that we
expect that when curvature is non-negative the mLSI constant is
controlled by the Gaussian concentration constant, and that in high
dimension the diameter is much larger than the Gaussian concentration
constant. Since $\pi_{K,L}$ is the uniform measure on all admissible
configurations, the distribution of the number of particles on a given
site is a binomial distribution, with parameters $K$ and $1/L$, so
that it satisfies an exponential concentration property with a
constant that only depends on the particle density $K/L$ (which
matches well with the result of Morris). For fixed density
$ \rho =K/L$ and large $K$ and $L$, the binomial law approximates a
Poisson law with parameter $\rho$, so that the invariant measure looks
like a product of Poisson measures, with an added constraint of fixed
total sum (which is $K$). The results of \cite{BW} then suggest that
we should expect the Gaussian concentration constant to behave like
$1/L$. With the way we defined the rates of the Markov chain (that
differs with the rate used in \cite{Mor06} by a factor $1/L$, this
leads us to expect the mLSI constant to behave like $1/L^2$ at fixed
density $K/L$, and suggests that our result is off by a factor
$1/(K\log L)$ (since at fixed density, the asymptotic behavior of $K$
and $L$ is the same).

\begin{proof} [Proof of Lemma \ref{lem_diam_zrp}]
  We need to show that
  \begin{align*}
  \cW(\delta_\eta,\delta_{\tilde\eta})\leq cK\sqrt{L\log L}
  \end{align*}
  for any $\eta,\tilde\eta\in\cX_{K,L}$ and a suitable constant $c$.
  For each pair $\eta,\tilde\eta$ we can find a sequence
  $\eta=\eta_1,\dots,\eta_n=\tilde\eta$ of length at most $K$ such
  that $\eta_i$ and $\eta_{i+1}$ differ only by the position of a
  single  particle.
  From the triangle inequality for $\mathcal{W}$, it is enough to show
  that $\cW(\delta_{\eta_i},\delta_{\eta_{i+1}})\leq c\sqrt{L\log L}$.
  But when looking at the movement of a single particle, the situation
  is the same as for a random walk on the complete graph with rate
  $1/L$. More precisely, we claim that
  \begin{align*}
    \cW(\delta_{\eta_i},\delta_{\eta_{i+1}})\leq\cW(\delta_x,\delta_y)\;,
  \end{align*}
  where the right-hand side is the transport distance between Dirac
  masses in point $x,y$ on the complete graph with $L$ sites and rates
  $1/L$. To see this, we can lift an optimal solution to the
  continuity equation $(\rho_t,\psi_t)$ on the complete graph
  connecting $\delta_x,\delta_y$ to a solution to the continuity
  equation $(\bar\rho_t,\bar\psi_t)$ on the state space of the zero
  range process connecting $\delta_{\eta_i}, \delta_{\eta_{i+1}}$ (see
  \cite[Lem.~3.14]{Ma12}, where such a lifting is carried out in
  detail for a comparison to the two-point space). So it is enough to
  show that the distance on the complete graph induced by the simple
  random walk with unit rate has diameter bounded by $c\sqrt{\log L}$
  (the change in speed changes the diameter by a factor $\sqrt{L}$).
  This diameter bound will follow from a general diameter bound in
  Proposition \ref{bonnet_myers} below. For simple random walk on the
  complete graph, the minimal mass of a point is given by $\pi_* =
  1/L$ and curvature is bounded from below by $1/2$.
\end{proof}

We conclude with a general estimate on the diameter of $(\cX,d_\cW)$
that can be seen as a discrete analogue to the Bonnet-Myers theorem in
Riemannian geometry.

\begin{proposition} \label{bonnet_myers} Assume that
  $\Ric(\cX,Q,\pi)\geq \kappa$ for $\kappa > 0$. Then for any
  $x,y\in\cX$ we have 
  \begin{align*}
    d_{\mathcal{W}}(x,y) \leq 2\sqrt{\frac{-\log
    \pi(x) - \log \pi(y)}{\kappa}}\;.
  \end{align*}
  Thus, the diameter of $(\cX,d_\cW)$ is bounded by
  $2\sqrt{\frac{-2\log \pi_*}{\kappa}}$, where $\pi_* := \inf\{\pi(x):x\in\cX\}$.
\end{proposition}

The dependence on $\pi_*$ might seem undesirable, but since we used no
upper bound on the dimension, we cannot expect the diameter bound to
depend only on $\kappa$. In the case of the discrete hyper-cube of
dimension $n$, we have $-\log \pi_* = n\log 2$, which is the correct
dependence on the dimension.

\begin{proof} [Proof of Proposition \ref{bonnet_myers}]
  From the convexity of the entropy \eqref{eq:ent-convex}, we have
  \begin{align*}
   0\leq   \mathcal{H}(\rho_\frac{1}{2}^{x,y}) \leq \frac12\mathcal{H}(\delta_x) + \frac12\mathcal{H}(\delta_y) - \frac{\kappa}{8}d_{\mathcal{W}}(x,y)^2\;,
  \end{align*}
  where $(\rho^{x,y}_t)_{t\in[0,1]}$ is the $\cW$-geodesic connecting
  $\delta_x$ to $\delta_y$. We then use that
  $\mathcal{H}(\delta_x) = -\log\pi(x)$ to conclude.
\end{proof}

\bibliographystyle{plain}
\bibliography{weak2strong}

\end{document}